\documentclass[12pt,psfig,reqno]{amsart}
\usepackage{mathrsfs}
\usepackage[colorlinks,
            linkcolor=cyan,
            anchorcolor=cyan,
            citecolor=cyan
            ]{hyperref}
\usepackage{txfonts}
\usepackage{amscd}
\usepackage{cite}
\usepackage{epsfig}
\usepackage{verbatim}
\usepackage{mathdots}
\usepackage{amssymb}
\usepackage{amsfonts}
\usepackage{amsbsy}
 \usepackage{graphicx}
 \usepackage{epstopdf}
 \usepackage{latexsym,amsmath,amssymb,amsthm,amsfonts,arydshln,enumerate,graphicx}
 \usepackage[usenames]{color} 
 \usepackage{cancel}
 \numberwithin{equation}{section}
 
\newtheorem{prop}{Proposition}[section]
\newtheorem{lem}[prop]{Lemma}

\newtheorem{defi}{Definition}[section]

\newtheorem{thm}[prop]{Theorem}

\newtheorem{exam}{Example}[section]
\newtheorem{rem}{Remark}[section]
\newtheorem{qn}{Question}[section]

\begin{document}
\baselineskip=17pt
\title{Integer tile and Spectrality of Cantor-Moran measures with equidifferent digit sets}


\author{Sha Wu and Yingqing Xiao$^*$}
\address{Sha Wu, School of Mathematics, Hunan University, Changsha, Hunan, 410082, P.R. China}
\email{shaw0821@163.com}
\address{Yingqing Xiao, School of Mathematics and Hunan Province Key Lab of Intelligent Information Processing and Applied Mathematics, Hunan University, Changsha, 410082, P.R. China}
\email{ouxyq@hnu.edu.cn}

\date{\today}
\thanks{ The research is supported in part by National Natural Science Foundation of China (Grant Nos. 12471073)
\\
$^*$Corresponding author.}
\begin{abstract}
Let  $\left\{b_{k}\right\}_{k=1}^{\infty}$ be a sequence of  integers with $|b_{k}|\geq2$ and $\left\{D_{k}\right\}_{k=1}^{\infty} $ be a sequence of equidifferent digit sets with $D_{k}=\left\{0,1, \cdots, N-1\right\}t_{k},$
where $N\geq2$ is a prime number and $\{t_{k}\}_{k=1}^{\infty}$ is bounded. In this paper, we  study the existence of the Cantor-Moran measure $\mu_{\{b_k\},\{D_k\}}$  and show that  $\mathbf{D}_k:=D_k\oplus b_{k} D_{k-1}\oplus b_{k}b_{k-1} D_{k-2}\oplus\cdots\oplus b_{k}b_{k-1}\cdots b_2D_{1}$ is an integer tile  for all $k\in\mathbb{N}^+$  if and only if $\mathbf{s}_i\neq\mathbf{s}_j$ for all $i\neq j\in\mathbb{N}^{+}$, where $\mathbf{s}_i$ is defined as the numbers of factor $N$ in $\frac{b_1b_2\cdots b_i}{Nt_i}$. Moreover, we prove that $\mathbf{D}_k$  being an integer tile for all $k\in\mathbb{N}^+$ is a necessary condition for the Cantor-Moran measure to be a spectral measure, and we  provide  an example to demonstrate  that it cannot become a sufficient condition. Furthermore, under some additional assumptions, we  establish that  the Cantor-Moran measure  to be a spectral measure   is equivalent to $\mathbf{D}_k$  being an integer tile for all $k\in\mathbb{N}^+$.
\end{abstract}
\maketitle
\section{\bf Introduction\label{sect.1}}
\subsection {\bf Cantor-Moran measures}
For a finite subset $ E \subset\mathbb{R}^{n}$, we define $\delta_E=\frac{1}{\#E}\sum_{e\in E}\delta_e$, where $\#E$  denotes  the cardinality of $E$ and $\delta_e$ is the Dirac point mass measure at $e$.
Let  $\left\{E_{k}\right\}_{k=1}^{\infty}$ be a sequence of finite subsets on  $\mathbb{R}^{n}$ and write
\begin{equation*}\label{eq(2.811-1)}
\mu_{k}=\delta_{E_1}\ast\delta_{E_2}
\ast\cdots\ast\delta_{E_k}
\end{equation*}
for each $k\geq1$, where $ \ast$ is the convolution sign. We say that $\mu_{k}$ converges weakly to $\mu $ if
$$\lim\limits_{k\rightarrow\infty}\int fd\mu_{k}=\int fd\mu$$
 for all $f\in C_{b}(\mathbb{R}^n)$, where  $C_{b}(\mathbb{R}^n)$ denotes the set of all bounded continuous functions on  $\mathbb{R}^n$. If  $\mu_{k}$  converges weakly to a Borel probability measure, then the weak limit
is called the infinite convolution of $\delta_{k}$ and denoted  by
 \begin{equation*}
\mu=\delta_{E_1}\ast\delta_{E_2}\ast\delta_{E_3}
\ast\cdots.
\end{equation*}
 A natural subsequent question is the following.
\begin{qn}
Under what conditions does  $\mu_{k}$  converge weakly to $\mu$ ?
\end{qn}
Using some results on Fourier transforms, Jessen and Wintner \cite{JW1935} develop a general theory of infinite convolutions and in particular their convergence theory. Convergence theory of infinite convolutions is completed at  \cite[Theorem 34]{JW1935}, where it is shown that the convergence problem of infinite convolutions is identical with the convergence problem of infinite series the terms of which are independent random variables as considered by Kolmogoroff \cite{K1928}. Based on this convergence  theory of infinite convolutions, Li {\it et al.} \cite{LMW-2022} gave a sufficient and necessary condition for the
existence of infinite convolutions when   $\left\{E_{k}\right\}_{k=1}^{\infty}\subset\mathbb{R}^{n}_{+}$, where $\mathbb{R}^{n}_{+}=[0, +\infty)^{n}$. Moreover, for the general case that $\left\{E_{k}\right\}_{k=1}^{\infty}\subset\mathbb{R}^{n}$, they also provided a sufficient condition. In this paper, we will introduce a special class of infinite convolution and study its convergence.

Let  $\left\{b_{k}\right\}_{k=1}^{\infty}$ be a sequence of  integers with  $| b_{k}|\geq2$ and $\left\{D_{k}\right\}_{k=1}^{\infty}$ be a sequence of  digit sets with $D_{k}\subset \mathbb{Z}$. Define
\begin{equation*}\label{eq(2.811)}
\mu_{k}=\delta_{b_{1}^{-1}D_1}\ast\delta_{b_{1}^{-1}b_{2}^{-1}D_2}
\ast\delta_{b_{1}^{-1}b_{2}^{-1} b_{3}^{-1}D_3}\ast\cdots\ast\delta_{b_{1}^{-1}b_{2}^{-1}\cdots b_{k}^{-1}D_k}
\end{equation*}
for $k\geq1$.
If  $\mu_{k}$ converges weakly to a Borel probability measure, then the weak limit is called Cantor-Moran  measure and denoted by $\mu_{\{b_k\},\{D_k\}}$. Moreover, the Cantor-Moran measure $\mu_{\{b_k\},\{D_k\}}$  is supported on  the set
\begin{equation*}
 K(b_{k}, D_{k}) =\left\{\sum_{k=1}^{\infty} \frac{ d_{k} }{b_{1} b_{2} \cdots b_{k}} : d_{k}\in D_{k},~k\geq1\right\},
 \end{equation*}
where the set $K(b_{k}, D_{k}) $ is usually called a Cantor-Moran set. In particular, in the case of  $b=b_{k}$ and $D=D_{k}$, we say that   $\mu_{b, D}$ is a self-similar measure and $K(b, D)$ is a self-similar  set.
Since then, the research related to Cantor-Moran measure has become an active research field, see \cite{AH14,DHL2013,FY97,HH17,LD2020,SS00,Tang2018,WDL2018}. Some researchers have also noticed the existence problem of Cantor-Moran measure and have given some related results. Recently, An {\it et al.}\cite{ALZ24} showed that the  Cantor-Moran measure $\mu_{\{b_k\},\{D'_k\}}$ exists if and only if $\sum_{k=1}^{\infty}\frac{N_{k}}{b_{1} b_{2} \cdots b_{k}}<\infty$, where  $\mu_{\{b_k\},\{D'_k\}}$ is generated by  an integer sequence $\{b_{k}\}_{k=1}^{\infty}$ with $b_k\geq 2$ and a sequence  of consecutive digit sets $\{D'_{k}:= \{0,1,\cdots,N_{k}-1\}\}_{k=1}^{\infty}$  with $N_{k}\geq2$.
The first purpose of this paper is to study the existence of Cantor-Moran measure for equidifferent digit sets,  which is a further study of the results of An {\it et al.}\cite{ALZ24}. We can state our first result as following.
\begin{thm}\label{th(1.1)}
Given a sequence of integer   $\{b_k\}_{k=1}^{\infty}$ with $|b_{k}|\geq 2$ and a sequence of integer digit sets  $\{D_{k}\}_{k=1}^{\infty}$, where $D_{k}=\left\{0,1, \cdots, N_{k}-1\right\}t_{k}$ with $N_{k}\geq 2$ and $|t_{k}|\geq 1$, if $\sum_{k=1}^{\infty}|\frac{N_{k}t_{k}}{b_{1} b_{2} \cdots b_{k}}|<\infty$, then $$
\mu_{k}=\delta_{b_{1}^{-1}D_1}\ast\delta_{b_{1}^{-1}b_{2}^{-1}D_2}
\ast\delta_{b_{1}^{-1}b_{2}^{-1} b_{3}^{-1}D_3}\ast\cdots\ast\delta_{b_{1}^{-1}b_{2}^{-1}\cdots b_{k}^{-1}D_k}
$$ converges weakly to a Borel probability measure. Moreover, if $b_{k}\geq2$ and $t_{k}\geq1$,  then the converse is also true.
\end{thm}
\subsection {\bf  Spectrality and integer tile}
For a  Borel probability measure $\mu$  on $\mathbb{R}^{n}$ with compact support $\Omega$, we  say that $\Lambda\subset \mathbb{R}^{n}$ is a spectrum of $\mu$ if
\begin{equation}\label{eq(1.000)}
\{e^{2\pi i<\lambda,x>}:\lambda\in \Lambda\}  \ \text{ forms an orthogonal basis for}  \ L^{2}(\mu).
\end{equation}
In this case, we  call $\mu$  a \emph{spectral measure}, and we also say  that $(\mu ,\Lambda)$ forms a \emph{spectral pair}. In particular, if $\mu$ is the normalized Lebesgue measure supported on  a  Borel set $\Omega$ such that \eqref{eq(1.000)} holds for some $\Lambda\subset \mathbb{R}^{n}$, then $\Omega$ is called  a \emph{spectral set}. It should be pointed out that the spectrum is by no means unique. For example, any translate of a spectrum is again a spectrum, but more radically different choices are also available.

A Borel set $\Omega \subset \mathbb{R}^{n}$ with positive measure is said to tile $\mathbb{R}^{n}$ by translations if
there exists a discrete set $L\subset \mathbb{R}^d $ such that
 $$\bigcup_{l \in L}(\Omega+l)=\mathbb{R}^n \quad \text { and } \quad m((\Omega+l_1) \cap (\Omega+l_2))=0 \  \text { for all }l_1 \neq l_2\in L,$$ where $m(\cdot)$ denotes the Lebesgue measure, and $L$ is called the tiling complement of $\Omega$.
For the unit cube $\bar{\Omega}=[0, 1]^{n}$, it is well known that $\bar{\Omega}$ is a spectral set and $\bar{\Omega}$ tile $\mathbb{R}^{n}$ by translations, and it is not difficult to  verify that the set $\mathbb{Z}^{n}$ is  a spectrum for  $\bar{\Omega}$ and also is a tiling complement of $\bar{\Omega}$.  A more specific conclusion is that $\Lambda$ is a spectrum of $\bar{\Omega}$ if and only if $\Lambda$ is a tiling complement of $\bar{\Omega}$\cite{IP98, LRW00}. The main interest in studying spectral sets comes from its mysterious connection to tiling, originally a conjecture proposed by  Fuglede \cite{Fug74}, and today known as the Fuglede Conjecture.

 \

\noindent\textbf{ The Fuglede Conjecture.} $\Omega\subset \mathbb{R}^n $ is a  spectral set  if and only if  it tiles $\mathbb{R}^n$  by translation.

\

This conjecture had baffled mathematicians studying spectral sets for many years. Until 2004, Tao \cite{Ta04} showed that there are spectral sets of dimension $n\geq5$ that are not tiles.  Afterwards, in dimensions $n\geq3$, counterexamples to both directions of the conjecture were found by Kolountzakis and Matolcsi \cite{KM061,KM062}. These counterexamples are composed of finitely many unit cubes in special arithmetic arrangements and are highly disconnected, but Greenfeld and Kolountzakis \cite{GK2023} recently  showed that  the conjecture is false in both directions for connected sets of sufficiently high dimensions. Until now, the conjecture is still open  in  dimensions $n=1$ and $2$ for  both directions, but fortunately, Lev and Matolcsi \cite{LM19} discovered that the conjecture holds in any dimension for a convex body.
For an integer $p\geq1$, the ring of integers modulo $p$ is denoted by $\mathbb{Z}_p := \mathbb{Z}/p\mathbb{Z}$.  We  also know that the conjecture
holds on $\mathbb{Z}_s$ \cite{L02, FFS16}, $\mathbb{Z}_s\times  \mathbb{Z}_s $ \cite{IMP17}  and  $\mathbb{Z}_{s^{n}t}$ with $n \geq 1$ \cite{MK17}, where $s, t$ are different primes. For more discussion on  the conjecture for cyclic groups, the reader can refer to \cite{KMSV22,M22, Ma20} etc.

We call a finite set $D\subset\mathbb{Z}$  an \emph{integer tile} if there exists $L\subset\mathbb{Z}$ such
that $D \oplus L= \mathbb{Z}_p$, where $\oplus$ denotes the direct sum. The spectrality of self-similar/Cantor-Moran measure is intricately linked to the integer tile  property of the digit set. In 2002, {\L}aba and  Wang \cite{LW02} proposed a far-reaching conjecture that   the self-similar measure $\mu_{b,D}$  is a spectral measure, then $ \alpha D $ is an integer tile for some $\alpha\in \mathbb{R}$.  For the four digit sets $D=\{0,a, b ,c\}\subset\mathbb{R}$, An {\it et al.} \cite{AHL22} indicated that the self-similar measure $\mu_{b,D}$ is a spectral measure, then $b\in\mathbb{Z}$ and $D\oplus b D \oplus \cdots \oplus  b^{k-1}D$ is  an integer tile  for all $k\in\mathbb{N}^+$.
For the consecutive digit sets $D'_{k}:= \{0,1,\cdots, N_{k}-1\}$,  An {\it et al.}\cite{ALZ24} showed Cantor-Moran measure $\mu_{\{b_k\},\{D'_k\}}$  is a spectral measure if and only if $D'_k\oplus b_{k} D'_{k-1}\oplus b_{k}b_{k-1} D'_{k-2}\oplus \cdots \oplus  b_{k}b_{k-1}\cdots b_2D'_{1}$ is an integer tile  for all $k\in\mathbb{N}^+$, and they raised the following question.
\begin{qn}\label{qn(1.2)}
 If Cantor-Moran measure $\mu_{\{b_k\},\{D_k\}}$  is a spectral measure, is the digit set $D_k+ b_{k} D_{k-1}+ b_{k}b_{k-1} D_{k-2}+\cdots + b_{k}b_{k-1}\cdots b_2D_{1}$ an integer tile  for all $k\in\mathbb{N}^+$ ?
\end{qn}
 In fact, the
converse of Question \ref{qn(1.2)} is not valid. We can use the following example to illustrate this point.
\begin{exam}\label{ex(1.0)}{\rm
Let $\tilde{D}_{1}=\{0,1,2\}$, $\tilde{D}_{k}=\{0,1,2\}4$ and $b_{1}=  b_{ k }=3$ for all $k\geq 2$. It is easy to  verify that   $\tilde{D}_k\oplus b_{k} \tilde{D}_{k-1}\oplus b_{k}b_{k-1} \tilde{D}_{k-2}\oplus\cdots\oplus b_{k}b_{k-1}\cdots b_2\tilde{D}_{1}$ is an integer tile  for all $k\in\mathbb{N}^+$, but we shows that $\mu_{\{b_k\},\{\tilde{D}_k\}}$ is not a spectral measure in \cite[Theorem 1.6]{WX2024}.}
\end{exam}

Inspired by Question \ref{qn(1.2)},  this paper focuses on investigating the spectrality and integer tile properties of the Cantor measure for a class of equidifferent digit sets. Let the Cantor-Moran measure
\begin{equation}\label{eq(1.2)}
\mu_{\{b_k\},\{D_k\}}=\delta_{b_{1}^{-1}D_1}\ast\delta_{b_{1}^{-1}b_{2}^{-1}D_2}
\ast\delta_{b_{1}^{-1}b_{2}^{-1} b_{3}^{-1}D_3}\ast\cdots
\end{equation}
be generated by an integer sequence $\{b_k\}_{k=1}^{\infty}$ with $|b_k|\geq 2$  and an integer  sequence of digit sets $\{D_k\}_{k=1}^{\infty}$
with $D_{k}=\left\{0,1, \cdots, N-1\right\}t_{k},$
where $N\geq2$ is a prime number and $\{t_{k}\}_{k=1}^{\infty}$ is bounded with $|t_{k}|\geq1$. In fact, Theorem \ref{th(1.1)} shows the Cantor-Moran measure $\mu_{\{b_k\},\{D_k\}}$ exists.
In order to more succinctly describe, throughout this paper we define
$$\tau_{N}(A)=\max\{k\in\mathbb{N}: N^{k}\mid A\}$$
for  $A\in\mathbb{Z}$,
and we write
\begin{equation}\label{eq(1.2--1)}\mathbf{s}_k:=\tau_{N}(b_1b_2\cdots b_k)-\tau_{N}(Nt_k).
\end{equation}
for all $k\geq1.$
Now, we give an equivalent condition for  $D_k$ to be an integer tile.

\begin{thm}\label{th(1.2-1)}
  Let the Cantor-Moran measure $\mu_{\{b_k\},\{D_k\}}$ be defined by  \eqref{eq(1.2)}. Then $D_k\oplus b_{k} D_{k-1}\oplus b_{k}b_{k-1} D_{k-2}\oplus\cdots\oplus b_{k}b_{k-1}\cdots b_2D_{1}$ is an integer tile  for each $k\in\mathbb{N}^+$  if and only if $\mathbf{s}_i\neq\mathbf{s}_j$ for all $i\neq j\in\mathbb{N}^{+}$, where $\mathbf{s}_i$ and $\mathbf{s}_j$ are defined by  \eqref{eq(1.2--1)}.
\end{thm}

In addition, we provided a positive answer to Question \ref{qn(1.2)} for the Cantor-Moran measure $\mu_{\{b_k\},\{D_k\}}$ defined by  \eqref{eq(1.2)}.
\begin{thm}\label{th(1.2)}
Let the Cantor-Moran measure $\mu_{\{b_k\},\{D_k\}}$ be defined by  \eqref{eq(1.2)}. If $\mu_{\{b_k\},\{D_k\}}$ is a spectral measure, then $\mathbf{s}_i\neq\mathbf{s}_j$ for all $i\neq j\in\mathbb{N}^{+}$ and $D_n\oplus b_{n} D_{n-1}\oplus b_{n}b_{n-1} D_{n-2}\oplus\cdots\oplus b_{n}b_{n-1}\cdots b_2D_{1}$ is an integer tile  for all $n\in\mathbb{N}^+$, where $\mathbf{s}_i$ and $\mathbf{s}_j$ are defined by  \eqref{eq(1.2--1)}..
\end{thm}
\begin{rem}{\rm The condition that $\{t_{k}\}_{k=1}^{\infty}$ is bounded is not needed in the proof of Theorem \ref{th(1.2-1)}. Moreover, Theorem \ref{th(1.2)} extends the result of  Deng and Li \cite{DL23} for the case $N=2$ into a more general form, but we adopt a different approach from they to prove it.}
\end{rem}

In the observation of Example \ref{ex(1.0)}, it is easy to get that the
converse of Theorem \ref{th(1.2)} is incorrect. Hence, the natural question is: {\em under what conditions does the converse of Theorem \ref{th(1.2)} hold?} In this paper, we refer to the results of Cao {\it et al.} \cite{DL23,CDLW24} and give an answer for this question.
\begin{thm}\label{th(1.4)}
Let the Cantor-Moran measure $\mu_{\{b_k\},\{D_k\}}$ be defined by  \eqref{eq(1.2)}. Suppose that there  exists an integer $m_0\geq1$  such that $|b_k |>(N-1) | t_k |$ for all $k\geq m_{0}$, then the following statements are equivalent.
\begin{enumerate}[(i)]
\item $\mu_{\{b_k\},\{D_k\}}$ is a spectral measure;
 \item $\mathbf{s}_i\neq\mathbf{s}_j$ for all $i\neq j\in\mathbb{N}^{+};$
 \item  $D_k\oplus b_{k} D_{k-1}\oplus b_{k}b_{k-1} D_{k-2}\oplus\cdots\oplus b_{k}b_{k-1}\cdots b_2D_{1}$ is an integer tile  for all $k\in\mathbb{N}^+$,
\end{enumerate}
where $\mathbf{s}_i$ and $\mathbf{s}_j$ are defined by  \eqref{eq(1.2--1)}.\end{thm}
\begin{rem} {\rm In the usual results, the spectrality of the Cantor-Moran measure are studied under the  integer Hadamard triple condition, while the above results avoid this condition to study the spectrality directly. This also means that we will have to face more challenges in constructing spectra.}
\end{rem}

\subsection {\bf Organization.}
In Section \ref{sect.2}, we  mainly prove Theorem \ref{th(1.1)}. We divide the proof of Theorem \ref{th(1.1)} into two parts (see Propositions \ref{prop(2.2-1-1)} and   \ref{prop(2.2-1-1-1)}).
In this process,  the convergence theorem of Jessen and Wintner (see Theorem \ref{th(2.1)}) is used to transform the proof of Propositions \ref{prop(2.2-1-1)} and   \ref{prop(2.2-1-1-1)}.

In Section \ref{sect.3}, we introduce some basic definitions, fix notation that will be used in this paper and discuss basic results about spectrality of measures. We  give an equivalent conditions for the integral tlie (see Theorem \ref{th(1.2-1)}), where we use a conclusion  of  Tijdeman about the direct sum  decomposition of two subsets (see Theorem \ref{th(4.1-1)}). Moreover, we prove Theorem \ref{th(1.2)} by simply going on to show that  $\mathbf{s}_i\neq\mathbf{s}_j$ for all $i\neq j\in\mathbb{N}^{+}$.

In Section \ref{sect.4}, we  focus on proving  $``(ii)\Longrightarrow (i)"$ of Theorem \ref{th(1.4)} and decompose this proof  into  the following two cases.

 \textbf{Case I:}
 There exists an infinite subsequence $\{k_n\}_{n=1}^{\infty}$ of $\mathbb{N}^+$
 such that $\min\{\mathbf{s}_j : j> k_n\}> \max\{\mathbf{s}_j : j\leq k_n\}$ for all $n\geq1$.

 \textbf{Case II:} There exists $k_0\in\mathbb{N}^+$ such that $\min\{\mathbf{s}_j : j> k\}< \max\{\mathbf{s}_j : j\leq k\}$ for all $k\geq k_0$.

The method we prove it is to construct an appropriate   $\Lambda=\bigcup_{n=1}^{\infty}\Lambda_{n}$ satisfying the conditions of Theorem \ref{prop(4.2.2)} in each case. At the end of this section,  some examples are given to show that  it is reasonable to divide the discussion into Case I and Case II.

\section{\bf Weak convergence of Cantor-Moran  measures\label{sect.2}}
Before discussing the weak convergence property of Cantor-Moran measures, we first give the convergence theorem of Jessen and Wintner \cite[Theorem 34]{JW1935} for infinite convolution in one dimension, which can be expressed as the following Theorem \ref{th(2.1)}. To illustrate Theorem  \ref{th(2.1)} more concisely, we first give some definitions.

Let $\omega_k$ be the Borel probability measures on $\mathbb{R}$. We define
$$c(\omega_k)=\int_{\mathbb{R}}xd\omega_k(x)  \ \ \ \text{and}\ \ \ M(\omega_k)=\int_{\mathbb{R}}(x-c(\omega_k))^{2}d\omega_k(x).$$
It is easy to check that $$M(\omega_k)=\int_{\mathbb{R}}x^2d\omega_k(x)-c(\omega_k)^{2}.$$
We define a new Borel probability measure $\omega_{k, r}$  by
\begin{equation}\label{eq(2.1)} \omega_{k, r} (E)=\omega_k(E\cap B(r))+\omega_k(\mathbb{R}\setminus B(r)) \delta_0(E)
\end{equation}
for every Borel subset $E \subset\mathbb{R}$, where $ \delta_0$ denotes the Dirac measure at $0$ and $B(r)$ denotes the closed ball with center at $0$ and radius $r$.

\begin{thm}\label{th(2.1)}\cite[Theorem 34]{JW1935}
With the above notations, let $\{\omega_k\}_{k=1}^{\infty}$ be a sequence of Borel probability measures on  $\mathbb{R}$. Fix a constant $r>0$, and let $\omega_{k,r}$ be defined by \eqref{eq(2.1)} for the measure $\omega_{k}, k\geq 1$. Then
the sequence of convolutions $\{\omega_1\ast \omega_2\ast\cdots\ast\omega_k\}_{k=1}^{\infty}$ converges weakly to a Borel probability measure if and only if the following three series all converge:
 \begin{equation}\label{eq(2.2)}  \sum_{k=1}^{\infty}\omega_{k}(\mathbb{R} \setminus B(r)), \  \ \ \ \  \sum_{k=1}^{\infty}c(\omega_{k,r})  \ \ \ \text{and} \ \ \ \sum_{k=1}^{\infty}M(\omega_{k,r}).
\end{equation}
\end{thm}
To facilitate our proof, we first give a simple but useful lemmal.
\begin{lem}\label{lm(2.1-1)}
Let $N\geq2$ be an integer, if $N-1>|M|$, then
 $$0<\min\left\{ 1-\frac{ 1+\left\lfloor|M|\right\rfloor}{N },  \ \  \frac{ 1+\left\lfloor|M|\right\rfloor}{N} \right\} < \left|\frac{N  }{M}\right|,$$
 where $\lfloor a\rfloor$ denotes the largest integer which is smaller or equal to $a$.
\end{lem}
\begin{proof}
Since $N-1> |M|$, we have $\frac{1
	+\left\lfloor \left|M  \right|\right\rfloor}{N}\leq\frac{1
	+\left|M  \right|}{N}<1<\left|\frac{N}{M}\right|$,
which means that $0< 1-\frac{1
	+\left\lfloor \left|M  \right|\right\rfloor}{N} \leq  \left|\frac{N}{M}\right| $ and $0<\frac{1+\left\lfloor|M|\right\rfloor}{N}< \left|\frac{N  }{M}\right|.$
\end{proof}
With the above full preparation, Theorem \ref{th(1.1)} will be divided into the following  Propositions \ref{prop(2.2-1-1)} and \ref{prop(2.2-1-1-1)} to prove.
\begin{prop}\label{prop(2.2-1-1)}
Given a sequence of integer  $\{b_k\}_{k=1}^{\infty}$ with $|b_{k}|\geq 2$ and a sequence of integer digit sets  $\{D_{k}\}_{k=1}^{\infty}$, where $D_{k}=\left\{0,1, \cdots, N_{k}-1\right\}t_{k}$ with $N_{k}\geq 2$ and $|t_{k}|\geq 1$. If $\sum_{k=1}^{\infty} \left|\frac{N_{k}t_{k}}{b_{1} b_{2} \cdots b_{k}}\right|<\infty$, then  $$
\mu_{k}=\delta_{b_{1}^{-1}D_1}\ast\delta_{b_{1}^{-1}b_{2}^{-1}D_2}
\ast\delta_{b_{1}^{-1}b_{2}^{-1} b_{3}^{-1}D_3}\ast\cdots\ast\delta_{b_{1}^{-1}b_{2}^{-1}\cdots b_{k}^{-1}D_k}
$$ converges weakly to a Borel probability measure.
\end{prop}
\begin{proof}
Write $\omega_{k}=\delta_{(b_1b_2\cdots b_k)^{-1}D_k}$ for all $k\geq1$. Let $r=1$, and $\omega_{k,r}$ be defined by \eqref{eq(2.1)} for all $k\geq1$.
By Theorem \ref{th(2.1)}, to prove \begin{align*}\mu_{k}&=\delta_{b_{1}^{-1}D_1}\ast\delta_{b_{1}^{-1}b_{2}^{-1}D_2}
\ast\delta_{b_{1}^{-1}b_{2}^{-1} b_{3}^{-1}D_3}\ast\cdots\ast\delta_{b_{1}^{-1}b_{2}^{-1}\cdots b_{k}^{-1}D_k}
\\&=\omega_{1}\ast\omega_{2}\ast\cdots\ast\omega_{k}
\end{align*}
converges weakly to a Borel probability measure, we just need to prove that  the three series of equation \eqref{eq(2.2)} all converge. In the following, we estimate each of these three series respectively.
According to some  simple calculations, we have

 $\begin{aligned}(1). \sum_{k=1}^{\infty}\omega_{k}(\mathbb{R} \setminus B(1)) = \sum_{k=1}^{\infty}\delta_{(b_1b_2\cdots b_k)^{-1}D_k}(\mathbb{R} \setminus B(1))
 &=\sum_{\{k : N_k -1>| b_1b_2\cdots b_k t_{k}^{-1}|\}} \left(1-\frac{1}{N_k}\left(1+\left\lfloor \left|\frac{b_1b_2\cdots b_k}{t_{k}}\right) \right|\right\rfloor\right).
\end{aligned}$
It follows from Lemma \ref{lm(2.1-1)} that
$$\sum_{k=1}^{\infty}\left|\omega_{k}(\mathbb{R} \setminus B(1))\right| <\sum_{\{k : N_k -1>| b_1b_2\cdots b_k t_{k}^{-1}|\}} \left|\frac{N_{k}t_{k}}{b_{1} b_{2} \cdots b_{k}}\right|.$$

$\begin{aligned} (2). \sum_{k=1}^{\infty}\left|c(\omega_{k,1})\right|&
= \sum_{k=1}^{\infty}\left| \int_{  B(1) }x d\delta_{(b_1b_2\cdots b_k)^{-1}D_k}(x)\right|
\\&= \sum_{\{k: N_k -1\leq| b_1b_2\cdots b_k t_{k}^{-1}|\}} \left|\sum_{d=0}^{N_{k}-1} \frac{dt_{k}}{N_kb_1b_2\cdots b_k }\right|+\sum_{\{k: N_k -1>| b_1b_2\cdots b_k t_{k}^{-1}|\}} \left|\sum_{d=0}^{\lfloor | b_1b_2\cdots b_k t_{k}^{-1}|\rfloor} \frac{dt_{k}}{N_kb_1b_2\cdots b_k }\right|
\\&=\sum_{\{k: N_k -1\leq| b_1b_2\cdots b_k t_{k}^{-1}|\}}  \frac{|t_{k}  |(N_{k}-1)}{2|b_1b_2\cdots b_k| }+\sum_{\{k: N_k -1>| b_1b_2\cdots b_k t_{k}^{-1}|\}}  \frac{ \left|t_{k}\right|\left(1+\left\lfloor \left|\frac{b_1b_2\cdots b_k}{t_{k}}\right|\right\rfloor\right)\left\lfloor \left|\frac{b_1b_2\cdots b_k}{t_{k}}\right|\right\rfloor}{2N_k|b_1b_2\cdots b_k|} .
\end{aligned}$
By a simple calculation and Lemma \ref{lm(2.1-1)}, we have
$$ \sum_{k=1}^{\infty}\left|c(\omega_{k,1})\right|<\sum_{k=1}^{\infty} \left|\frac{N_{k}t_{k}}{b_{1} b_{2} \cdots b_{k}}\right|.
$$

 $\begin{aligned}(3).
\sum_{k=1}^{\infty}\left|M(\omega_{k,1}\right|=\sum_{k=1}^{\infty} \int_{\mathbb{R}}(x-c(\omega_{k,1}))^{2}d\omega_{k,1}(x)
 =\sum_{k=1}^{\infty} \left(\int_{\mathbb{R}}x^2 d\omega_{k,1}(x) - c(\omega_{k,1})^2\right)  \leq\sum_{k=1}^{\infty}\int_{\mathbb{R}}\left|x\right| d\omega_{k,1}(x).
\end{aligned}$
Similar to  (2), it can be concluded that  $$\sum_{k=1}^{\infty}|M(\omega_{k,1})|<\sum_{k=1}^{\infty} \left|\frac{N_{k}t_{k}}{b_{1} b_{2} \cdots b_{k}}\right|.$$

 Combining (1) - (3) with $\sum_{k=1}^{\infty}|\frac{N_{k}t_{k}}{b_{1} b_{2} \cdots b_{k}}|<\infty$,
we have $\sum_{k=1}^{\infty}\omega_{k}(\mathbb{R} \setminus B(1))$,
$\sum_{k=1}^{\infty}c(\omega_{k,1})$ and $\sum_{k=1}^{\infty}M(\omega_{k,1})$  all converge.
\end{proof}
\begin{prop}\label{prop(2.2-1-1-1)}
Given a sequence of integers  $\{b_k\}_{k=1}^{\infty}$ with $b_{k}\geq 2$ and a sequence of integer digit sets  $\{D_{k}\}_{k=1}^{\infty}$, where $D_{k}=\left\{0,1, \cdots, N_{k}-1\right\}t_{k}$ with $N_{k}\geq 2$ and $t_{k}\geq 1$. If $$
\mu_{k}=\delta_{b_{1}^{-1}D_1}\ast\delta_{b_{1}^{-1}b_{2}^{-1}D_2}
\ast\delta_{b_{1}^{-1}b_{2}^{-1} b_{3}^{-1}D_3}\ast\cdots\ast\delta_{b_{1}^{-1}b_{2}^{-1}\cdots b_{k}^{-1}D_k}
$$ converges weakly to a Borel probability measure, then  $\sum_{k=1}^{\infty} \frac{N_{k}t_{k}}{b_{1} b_{2} \cdots b_{k}} <\infty$.
\end{prop}
\begin{proof}
Suppose $\mu_{k}$ converges weakly to a Borel probability measure, by Theorem \ref{th(2.1)}, we have
$$\sum_{k=1}^{\infty}\omega_{k}(\mathbb{R} \setminus B(1)) <\infty,  \ \ \ \ \sum_{k=1}^{\infty}c(\omega_{k,1})<\infty, \ \ \ \ \sum_{k=1}^{\infty}M(\omega_{k,1})<\infty.$$
According to the proof of Proposition \ref{prop(2.2-1-1)}, we have
\begin{align} \label{1.10005}
\sum_{\{k: N_k -1>  b_1b_2\cdots b_k t_{k}^{-1} \}} \left(1-\frac{1}{N_k} \left(1+\left\lfloor  \frac{b_1b_2\cdots b_k}{t_{k}}\right) \right\rfloor \right)<\infty
\end{align}
and
$$
\sum_{\{k: N_k -1\leq  b_1b_2\cdots b_k t_{k}^{-1} \}}  \frac{t_{k}(N_{k}-1)}{b_1b_2\cdots b_k }+\sum_{\{k: N_k -1>  b_1b_2\cdots b_k t_{k}^{-1} \}} \frac{ t_{k}\left(1+\left\lfloor \frac{b_1b_2\cdots b_k}{t_{k}}  \right\rfloor\right)\left\lfloor  \frac{b_1b_2\cdots b_k}{t_{k}}   \right\rfloor}{N_kb_1b_2\cdots b_k }<\infty.
$$
This means that
\begin{align} \label{1.10007}
\sum_{\{k: N_k -1\leq b_1b_2\cdots b_k t_{k}^{-1} \}}  \frac{t_{k}(N_{k}-1)}{b_1b_2\cdots b_k }<\infty
\ \ \  \text{and} \ \ \
\sum_{\{k: N_k -1>  b_1b_2\cdots b_k t_{k}^{-1} \}} \frac{ t_{k}\left(1+\left\lfloor \frac{b_1b_2\cdots b_k}{t_{k}}  \right\rfloor\right)\left\lfloor \frac{b_1b_2\cdots b_k}{t_{k}}  \right\rfloor}{N_kb_1b_2\cdots b_k }<\infty.
\end{align}
It follows from the fact  $\frac{b_1b_2\cdots b_k}{t_{k}}< 1+\left\lfloor  \frac{b_1b_2\cdots b_k}{t_{k}}  \right\rfloor $ and \eqref{1.10007}  that
\begin{align} \label{1.10008}\sum_{\{k: N_k -1> b_1b_2\cdots b_k t_{k}^{-1}\}} \frac{1}{N_k}\left\lfloor\frac{ b_1b_2\cdots b_k}{t_{k}}  \right\rfloor<\sum_{\{k: N_k -1> b_1b_2\cdots b_k t_{k}^{-1}\}} \frac{ t_{k}\left(1+\left\lfloor \frac{b_1b_2\cdots b_k}{t_{k}} \right\rfloor\right)\left\lfloor  \frac{b_1b_2\cdots b_k}{t_{k}}  \right\rfloor}{N_kb_1b_2\cdots b_k }<\infty,
\end{align}
and we conclude from \eqref{1.10005}, \eqref{1.10008}  that
$$\begin{aligned}\sum_{\{k: N_k -1>  b_1b_2\cdots b_k t_{k}^{-1} \}}  \frac{1}{2}&\leq \sum_{\{k: N_k -1>  b_1b_2\cdots b_k t_{k}^{-1} \}} \left(1-\frac{1}{N_k}\right)\\&=\sum_{\{k: N_k -1>  b_1b_2\cdots b_k t_{k}^{-1} \}} \left(1-\frac{1}{N_k} \left(1+\left\lfloor  \frac{b_1b_2\cdots b_k}{t_{k}} \right\rfloor \right)\right) + \sum_{\{k: N_k -1>  b_1b_2\cdots b_k t_{k}^{-1} \}} \frac{1}{N_k} \left\lfloor  \frac{b_1b_2\cdots b_k}{t_{k}} \right\rfloor \\&<\infty.
\end{aligned}$$
This implies that $\#\{k: N_k -1>  b_1b_2\cdots b_k t_{k}^{-1} \}<\infty$. For this reason, we have $\sum_{\{k: N_k -1>  b_1b_2\cdots b_k t_{k}^{-1} \}} \frac{N_{k}t_{k}}{b_{1} b_{2} \cdots b_{k}}<\infty$. Combining this with \eqref{1.10007},
 it is easy to deduce
 $$\begin{aligned}
 \sum_{k=1}^{\infty} \frac{N_{k}t_{k}}{b_{1} b_{2} \cdots b_{k}} &=\sum_{\{k: N_k -1\leq b_1b_2\cdots b_k t_{k}^{-1} \}} \frac{N_{k}t_{k}}{b_{1} b_{2} \cdots b_{k}}+\sum_{\{k: N_k -1>  b_1b_2\cdots b_k t_{k}^{-1} \}} \frac{N_{k}t_{k}}{b_{1} b_{2} \cdots b_{k}}
 \\&\leq \sum_{\{k: N_k -1\leq b_1b_2\cdots b_k t_{k}^{-1}\}}  \frac{2(N_{k}-1)t_{k}}{b_1b_2\cdots b_k } +\sum_{\{k: N_k -1>  b_1b_2\cdots b_k t_{k}^{-1} \}} \frac{N_{k}t_{k}}{b_{1} b_{2} \cdots b_{k}}
 \\&<\infty,
\end{aligned}$$
and the proof is complete.
\end{proof}
\begin{proof}[\bf Proof of Theorem \ref{th(1.1)}]
The proof can be derived from Propositions \ref{prop(2.2-1-1)} and  \ref{prop(2.2-1-1-1)}.
 \end{proof}

\section{\bf Spectrality of Cantor-Moran measures  \label{sect.3}}

\subsection {\bf Preliminary\label{sect.3.1}}
Let $\mu$ be a Borel probability measure on $ \mathbb{R}$. The Fourier transform of $\mu$ is defined by
$$
\hat{\mu} (x)=\int e^{2\pi i x \xi }d\mu(\xi)
 \ \ \text{for} \ \ x \in\mathbb{R}.
$$
Denote  $\mathcal{Z}(\hat{\mu} ):= \{\xi \in \mathbb{R}: \hat{\mu}(\xi)=0\} $ to be the zero set of $\hat{\mu}$. For a countable discrete set $\Lambda\subset \mathbb{R}$, it is easy to see
that $E(\Lambda)=\{e^{2\pi i\lambda x}:\lambda\in \Lambda\}$  is an orthogonal family of $L^{2}(\mu)$ if and only if
$$
0=\langle e^{2\pi i \lambda_1 x},e^{2\pi i\lambda_2 x}\rangle_{L^2(\mu)}=\int e^{2\pi i (\lambda_1-\lambda_2) x}d\mu=\hat{\mu}(\lambda_1-\lambda_2)
$$
for any $\lambda_1\neq\lambda_2\in\Lambda$. Therefore, the orthogonality of $E(\Lambda)$ is equivalent
to
\begin{equation}\label {eq(2.6)}
 (\Lambda-\Lambda)\setminus \{0\} \subset\mathcal{Z}(\hat{\mu}).
\end{equation}
In this case, we call $\Lambda$  an {\it orthogonal set} (respectively, {\it spectrum}) of $\mu$ if $ E(\Lambda)$ forms an orthogonal system (respectively, orthogonal basis) for $L^2(\mu)$. Define
$$ Q_{\Lambda}(x)=\sum_{\lambda \in \Lambda}|\hat{\mu}(x+\lambda)|^{2} \ \ \text{for} \ \ x \in\mathbb{R}.$$
In \cite[Lemma 4.2]{JP98}, Jorgensen and Pedersen  given a criterion that allows us to determine whether a countable set  $\Lambda$  is an  orthogonal set or a spectrum of $\mu$.
\begin{prop}[\cite{JP98}]\label{prop(2.3)}
Let $\mu$ be a Borel probability measure with compact support and   $\Lambda \subset \mathbb{R}$ be countable set. Then
\begin{enumerate}[\rm(i).]
 \item $\Lambda$ is an  orthogonal set of $\mu$ if and only if $Q_{\Lambda}(x) \leq 1$ for $x \in \mathbb{R}$.
 \item $\Lambda$ is a spectrum of $\mu$ if and only if $Q_{\Lambda}(x) \equiv 1$ for $x\in \mathbb{R}$.
\end{enumerate}
Moreover,  if $\Lambda$ is an  orthogonal set, then $Q_{\Lambda}(x)$ is an entire function.
\end{prop}

The following lemma gives an efficient method for discriminating that a countable set $\Lambda$ is not a spectrum of measure $\mu$.
\begin{lem}[\cite{DHL2014}]\label{lem01}Let $\mu=\mu_{1}\ast  \mu_{2}$ be the convolution of two probability measures $\mu_{i}(i=1,2)$, and they are not Dirac measures. Suppose that $\Lambda$ is a  orthogonal set of $\mu_{1}$, then $\Lambda$ is also a orthogonal set of $\mu$, but cannot be a spectrum of $\mu$.
\end{lem}

\subsection {\bf Proof of  Theorems  \ref{th(1.2-1)}\label{sect.3.2}}
Given an integer sequence  $\{b_k\}_{k=1}^{\infty}$ with $|b_k|\geq 2$ and a sequence of digit sets $\{D_{k}=\left\{0,1, \cdots, N-1\right\}t_{k}\}_{k=1}^{\infty}$ with $|t_k|\geq1$, where  $N\geq2$ is a prime and $\{t_{k}\}_{k=1}^{\infty}$ is bounded, it follows from Theorem \ref{th(1.1)} that
 $\mu_{k}:=\delta_{b_{ 1}^{-1}D_{ 1}}\ast\delta_{b_{ 1}^{-1}b_{ 2}^{-1}D_{ 2}}
\ast\cdots\ast \delta_{b_{ 1}^{-1}b_{ 2}^{-1}\cdots b_{ k}^{-1} D_{k }}$
converges weakly to the Cantor-Moran measure
\begin{equation}\label{eq(2.08)}\begin{aligned}
 \mu_{\{b_k\},\{D_k\}}:=\delta_{b_{1}^{-1}D_1}\ast\delta_{b_{1}^{-1}b_{2}^{-1}D_2}
\ast\delta_{b_{1}^{-1}b_{2}^{-1} b_{3}^{-1}D_3}\ast\cdots\end{aligned}.
\end{equation}
Recall that
\begin{equation}\label{eq(3.6.1)}
\mathbf{s}_k=\tau_{N}(b_1b_2\cdots b_k)-\tau_{N}(Nt_k)
\end{equation}  for all $k\geq1$,
where  $\tau_{N}\left(A\right)=\max\left\{k\in\mathbb{N}: N^{k}\mid A\right\}$
for an integer $A$.  For convenience, we use $\tau$ to represent $\tau_{N}$ in the following proof.

Before proving Theorem \ref{th(1.2-1)}, we first give the following theorem, which is given by Tijdeman\cite{T} and plays an important role in the proof of Theorem \ref{th(1.2-1)}.
\begin{thm}\label{th(4.1-1)}\cite{T}
Let $D$ be  finite, and let $D\oplus L=\mathbb{Z}$, $0\in D\cap L$. Suppose that $\gcd(l, \#D)=1$, then $lD\oplus L=\mathbb{Z}$.
\end{thm}

\begin{proof}[\bf Proof of  Theorem \ref{th(1.2-1)}]
We first prove the sufficiency. For any  $k\in\mathbb{N}^+$, let $\alpha_i= \tau (b_1b_2\cdots b_k)-1-\mathbf{s}_i$ for $1\leq i\leq k$.  Since $\mathbf{s}_i\neq\mathbf{s}_j$ for all $i\neq j\in\mathbb{N}^{+}$, we have $\alpha_i\neq \alpha_j\geq0$ for all $i\neq j\in\{1,2, \cdots, k\}$ and
$$\begin{aligned}\bar{D}_k&=D_k+b_{k} D_{k-1}+b_{k}b_{k-1} D_{k-2}+b_{k}b_{k-1}\cdots b_2D_{1}
\\&=N^{\alpha_k}\left\{0,1, \cdots, N-1\right\}l_{\alpha_k}+N^{\alpha_{k-1}}\left\{0,1, \cdots, N-1\right\}l_{\alpha_{k-1}}+\cdots+N^{\alpha_{1}}\left\{0,1, \cdots, N-1\right\}l_{\alpha_1}
\end{aligned}$$
 for some $l_{\alpha_i}\in \mathbb{Z}\setminus N\mathbb{Z}$.
For convenience, we rearrange $\alpha_k, \alpha_{k-1}, \cdots, \alpha_{1}$ so that $ \alpha_k>\alpha_{k-1}>\cdots>\alpha_1.$ Let $L_{k}=\oplus_{i=0}^{k-1}L_{i, i+1}$,
where
$$L_{0, 1}= \begin{cases}\left\{0\right\}, & if \ \ \alpha_{1}= 0;\\ \oplus_{j=0}^{\alpha_{1}-1}N^{j}\left\{0,1, \cdots,N-1\right\}, & if \ \ \alpha_{1}>0;\end{cases}$$
and
$$
L_{i, i+1}= \begin{cases}\left\{0\right\}, & if \ \ \alpha_{i+1}= \alpha_{i}+1;\\ \oplus_{j=\alpha_{i}+1}^{\alpha_{i+1}-1}N^{j}\left\{0,1, \cdots,N-1\right\}, & if \ \ \alpha_{i +1}>\alpha_{i}+1;\end{cases}
$$
for $1\leq i\leq k-1$.

{\rm {\bf Claim 1.} The representation of the elements in  $\bar{D}_k + L_{k}$ is unique, and $ x_2-x_1 \notin N^{\alpha_k +1}\mathbb{Z}$ for any  $x_1\neq x_2\in \bar{D}_k + L_{k}$.}
\begin{proof}[\bf Proof of  Claim 1] We show that the expression is unique, and  $ x_2 -x_1 \notin N^{\alpha_k +1}\mathbb{Z}$ can be proved similarly.
Suppose that there are two distinct sequences $ \{z_j\}_{z=0}^{\alpha_k}$ and $ \{z_j'\}_{j=0}^{\alpha_k}$ with $z_j,  z_j' \in \left\{0,1, \cdots, N-1\right\}$ such that  $$\sum_{j=0}^{\alpha_k}N^{j}z_jl_j=\sum_{j=0}^{\alpha_k}N^{j}z_j'l_j\in\bar{D}_k + L_{k},$$
where $l_j=1$ for $j\in\{0,1,\cdots,\alpha_k\}\setminus \{\alpha_1,\alpha_2,\cdots,\alpha_k\}$.
Then $\sum_{j=0}^{\alpha_k}N^{j}(z_j'-z_j)l_j=0$. Let $0\leq t\leq \alpha_k$ be the smallest integer such that $z_t\neq z_t'$, we have
\begin{equation}\label{eq4.2----1} (z_t'-z_t)l_t=\sum_{j=0}^{\alpha_k-t}N^{j+1}(z_j-z_j')l_j\in N\mathbb{Z}.\end{equation}
Since  $\gcd(l_t, N)=1$, it follows from \eqref{eq4.2----1} that $z_t'-z_t\in N\mathbb{Z}$, which contradicts the fact $z_{t}-z'_{t}\in\pm\{1, 2, \cdots,N-1\}.$ The claim follows, i.e., $ \bar{D}_k + L_{k}=\bar{D}_k \oplus L_{k}$.
\end{proof}
Since $\#(\bar{D}_k\oplus L_{k})=N^{\alpha_k +1}$, it follows from Claim 1 that $\bar{D}_k \oplus L_{k}$  is  a complete residue system of $N^{\alpha_k +1}$. Hence, $\bar{D}_k \oplus L_{k}=\mathbb{Z}_{N^{\alpha_k +1}}$, i.e.,  $D_k\oplus b_{k} D_{k-1}\oplus b_{k}b_{k-1} D_{k-2}\oplus \cdots \oplus b_{k}b_{k-1}\cdots b_2D_{1}$ is an integer tile.

Next, we prove the  necessity by contradiction. Suppose that there exist $j_0>i_0$ such that $\mathbf{s}_{i_0}=\mathbf{s}_{j_0}$ and $D_k\oplus b_{k} D_{k-1}\oplus b_{k}b_{k-1} D_{k-2}\oplus \cdots\oplus b_{k}b_{k-1}\cdots b_2D_{1}$ is an integer tile  for all $k\in\mathbb{N}^+$. Then there exists $L\subset\mathbb{Z}$ such that
\begin{equation}\label{eq(3.5--22)}\begin{aligned}D_{j_0}\oplus b_{j_0} D_{j_{0}-1}\oplus \cdots\oplus b_{j_{0}}b_{j_{0}-1}\cdots b_{i_{0}+1}D_{i_0}\oplus  \cdots \oplus b_{j_{0}}b_{j_{0}-1}\cdots b_{2}D_{1}\oplus L=\mathbb{Z}.
\end{aligned}\end{equation}
Let $\alpha= \tau (b_1b_2\cdots b_{j_0})-1-\mathbf{s}_{i_0}$, then there exist $l_{i_0}, l_{j_0}\in \mathbb{Z}\setminus N\mathbb{Z}$ such that $ b_{j_{0}}b_{j_{0}-1}\cdots b_{i_{0}+1}D_{i_0}=N^{\alpha}\left\{0,1, \cdots, N-1\right\}l_{i_0}$ and $ D_{j_0}=N^{\alpha}\left\{0,1, \cdots, N-1\right\}l_{j_0}$. Writing
$$\bar{L}:= b_{j_0} D_{j_{0}-1}\oplus \cdots\oplus b_{j_{0}}b_{j_{0}-1}\cdots b_{i_{0}+2}D_{i_0 +1}\oplus b_{j_{0}}b_{j_{0}-1}\cdots b_{i_{0}}D_{i_0 -1}\oplus \cdots\oplus b_{j_{0}}b_{j_{0}-1}\cdots b_{2}D_{1}\oplus L.$$
According to  \eqref{eq(3.5--22)}, we have
$$ N^{\alpha}\left\{0,1, \cdots, N-1\right\}l_{j_0}\oplus N^{\alpha}\left\{0,1, \cdots, N-1\right\}l_{i_0}\oplus\bar{L}=\mathbb{Z}.$$
Since $\gcd(l_{i_0}, N)=1$ and $\gcd(l_{j_0}, N)=1$, it follows from Theorem
\ref{th(4.1-1)} that
$$ N^{\alpha}\left\{0,1, \cdots, N-1\right\}l_{j_0}l_{i_0}\oplus N^{\alpha}\left\{0,1, \cdots, N-1\right\}l_{i_0}l_{j_0}\oplus\bar{L}=\mathbb{Z},$$
which contradicts the definition of a direct sum.

Therefore, this  theorem is proved.
\end{proof}

In order to facilitate our subsequent proof, we give a very useful lemma here.
\begin{lem}\label{rem(1-1)}
 For any prime $N\geq2$ and integer $k\geq2$, if $a_i\neq a_j$ for all $1\leq i\neq j\leq k$ and $c_l\in\mathbb{Z}\setminus N\mathbb{Z}$ for all $1\leq l\leq k$, we have
$W:= N^{a_1}c_{ 1 } \{0, 1, \cdots, N-1\} + N^{a_2}c_{ 1 } \{0, 1, \cdots, N-1\} +\cdots+ N^{a_k}c_{ k } \{0, 1, \cdots, N-1\}$ is a direct sum, i.e., $W=\oplus_{i=1}^{k} N^{a_i}c_{ i } \{0, 1, \cdots, N-1\}  $.
\end{lem}
\begin{proof}
This proof is the same as Claim 1, so we omit its proof.
\end{proof}
\subsection {\bf Proof of  Theorems  \ref{th(1.2)}\label{sect.3.3}}
Since the sequence $\left\{t_{k}\right\}_{k=1}^{\infty}$ is bounded, according to \cite[Lemma 2.4 and Proposition 2.6]{WX2024}, we can always assume that $b_{k}\geq2$, $t_{k}\geq1$  and $Nt_{k}|b_1 $ for all $k\geq1$ in the following study of the spectrality of $\mu_{\{b_k\},\{D_k\}}$. For this reason, in all that follows, we assume that $ \mathbf{s}_k\geq0$.

Define $b'_{k}=\frac{ b_k}{N^{\tau(b_k)}}$, $t'_{k}=\frac{ t_k}{N^{\tau(t_k)}}$ and $\mathbf{b}_k:=b'_1b'_2\cdots b'_k$ for all $k\geq1$.
Then we have $ \frac{b_1b_2\cdots b_k}{Nt_k}=\frac{N^{\mathbf{s}_k}\mathbf{b}_k}{t'_k} $ and
\begin{equation}\label{eq(3.5)}
\mathcal{Z}(\hat{\mu}_{\{b_k\},\{D_k\}})
=\bigcup_{k=1}^{\infty}\mathcal{Z}(\hat{\delta}_{b_{1}^{-1}b_{2}^{-1} \cdots b_{k}^{-1}D_{k}})
=\bigcup_{k=1}^{\infty}\frac{N^{\mathbf{s}_k}\mathbf{b}_k(\mathbb{Z}\setminus N\mathbb{Z})}{t'_k}.
\end{equation}
Let \begin{equation}\label{eq(3.7)}\mathbf{s}=\min\{\mathbf{s}_k: k\in\mathbb{N}^{+}\}\ \ \ \ \text{and}\ \ \ \  \mathbf{A} =\{k : \mathbf{s}_k=\mathbf{s}\ \text{for}\ k\in\mathbb{N}^{+}\}.\end{equation}
By \eqref{eq(3.5)}, we have \begin{equation}\label{eq(3.8)}
\mathcal{Z}(\hat{\mu}_{\{b_k\},\{D_k\}})=\left(\bigcup_{k\in\mathbf{A}} \frac{N^{\mathbf{s}}\mathbf{b}_k(\mathbb{Z}\setminus N\mathbb{Z})}{t'_k}\right)\bigcup\left(\bigcup_{k\in\mathbb{N}^{+}\setminus\mathbf{A} } \frac{N^{\mathbf{s}_k}\mathbf{b}_k(\mathbb{Z}\setminus N\mathbb{Z})}{t'_k}\right)
\end{equation}
and $\mathbf{s}_k>\mathbf{s} $ for all $k\in\mathbb{N}^{+}\setminus\mathbf{A}$.

\begin{lem}\label{lem(3.1)}
Let $\mu_{\{b_k\},\{D_k\}}$ and $\mathbf{A} $ be defined by \eqref{eq(2.08)} and \eqref{eq(3.7)}, respectively. If $\mu_{\{b_k\},\{D_k\}}$ is a spectral measure, then  $\mathbf{A}$ is a finite set.
\end{lem}
\begin{proof}
Suppose, on the contrary, $\mathbf{A}$ is an infinite set. Since the sequence $\left\{t_{k}\right\}_{k=1}^{\infty}$ is bounded, there exist $i_1\neq i_2$ such that $t_{i_1}=t_{ i_2}$ and $\mathbf{s}_{i_1}=\mathbf{s}_{ i_2}=\mathbf{s}$, where $\mathbf{s}$ is defined by \eqref{eq(3.7)}. In view of \eqref{eq(3.5)}, we have $$\mathcal{Z}(\hat{\delta}_{b^{-1} b_{1}^{-1}b_{2}^{-1} \cdots b_{i_1}^{-1}D_{i_1}})= \frac{N^{\mathbf{s}}\mathbf{b}_{i_1}(\mathbb{Z}\setminus N\mathbb{Z})}{t'_{i_1}} \ \ \ \ \text{and} \ \ \ \
\mathcal{Z}(\hat{\delta}_{b^{-1} b_{1}^{-1}b_{2}^{-1} \cdots b_{i_2}^{-1}D_{i_2}})= \frac{N^{\mathbf{s}}\mathbf{b}_{i_2}(\mathbb{Z}\setminus N\mathbb{Z})}{t'_{i_2}}$$
and $t'_{i_1}=t'_{i_2}$ since $t_{i_1}=t_{ i_2}$ . Without loss of generality, we can assume that $i_1<i_2$. Combining this with the definition of  $\mathbf{b}_{i_1}$ and $\mathbf{b}_{i_2}$, we have $\mathcal{Z}(\hat{\delta}_{b^{-1} b_{1}^{-1}b_{2}^{-1} \cdots b_{i_2}^{-1}D_{i_2}})\subset\mathcal{Z}(\hat{\delta}_{b^{-1} b_{1}^{-1}b_{2}^{-1} \cdots b_{i_1}^{-1}D_{i_1}})$. Let $\nu =\ast_{k\in\mathbb{N}^{+}\setminus\{i_2\}}\delta_{b_{1}^{-1}b_{2}^{-1} \cdots b_{k}^{-1}D_k}$, then $ \mu_{\{b_k\},\{D_k\}}= \delta_{b^{-1} b_{1}^{-1}b_{2}^{-1} \cdots b_{i_2}^{-1}D_{i_2}} \ast\nu.$ Hence, $ \mathcal{Z}(\hat{\mu}_{\{b_k\},\{D_k\}})=\mathcal{Z}(\hat{\nu} )$. According to Lemma \ref{lem01}, $\mu_{\{b_k\},\{D_k\}}$ is not a spectral measure, which is a contradiction.
\end{proof}
The following proposition establishes the relationship on spectrality among the three measures $\mu_1$, $\mu_2$ and $ \mu_1\ast \mu_2$.
\begin{prop}\cite[Theorem 3.3]{Wu2024}\label{prop(4.10-1)}
Let $\mu=\mu_1\ast \mu_2$, where the support of $\mu_1$ is a finite set and $\mu_2$ is a periodic function. If $\mu $ is a spectral measure and satisfies the following $(*)$ condition
$$(*): \ \  If \ \ \lambda_{1}, \lambda_{2} \in \mathcal{Z}(\hat{\mu}_2) \backslash \mathcal{Z}(\hat{\mu}_1) \ \  and  \ \ \lambda_{1}-\lambda_{2} \in \mathcal{Z}(\hat{\mu}), \ \
  then  \ \  \lambda_{1}-\lambda_{2} \in \mathcal{Z}(\hat{\mu}_2) \backslash \mathcal{Z}(\hat{\mu}_1).$$
  Then  both $ \mu_1$ and $\mu_1$ are spectral measures.
\end{prop}

Define \begin{equation}\label{eq(3.9)}\omega_1=\ast_{k\in\mathbf{A}}\delta_{b_{1}^{-1}b_{2}^{-1} \cdots b_{k}^{-1}D_k} \ \ \ \text{and} \ \ \ \omega_2=\ast_{k\in\mathbb{N}^{+}\setminus\mathbf{A}}\delta_{b_{1}^{-1}b_{2}^{-1} \cdots b_{k}^{-1}D_k}.\end{equation}
Hence, \eqref{eq(2.08)} can be expressed as  $\mu_{\{b_k\},\{D_k\}}=\omega_1\ast \omega_2$.
\begin{lem}\label{lem(3.2)}
 Let $\mu_{\{b_k\},\{D_k\}}:=\omega_1\ast \omega_2$ be defined by \eqref{eq(3.9)}, and let $\left\{\lambda_{1}, \lambda_{2}\right\}$ be any orthogonal set  of $\mu_{\{b_k\},\{D_k\}}$.
 If  $\lambda_{1}, \lambda_{2} \in \mathcal{Z}(\hat{\omega}_2) \backslash \mathcal{Z}(\hat{\omega}_1)$, then $\lambda_{1}-\lambda_{2} \in \mathcal{Z}(\hat{\omega}_2) \backslash \mathcal{Z}(\hat{\omega}_1)$.
\end{lem}
\begin{proof}
This proof is easy to verify by \eqref{eq(3.8)}, so we omit it here.
\end{proof}

\begin{prop}\label{prop(3.3)}
Let $\mu_{\{b_k\},\{D_k\}}=\omega_1\ast \omega_2$ be defined by \eqref{eq(3.9)}. If $\mu_{\{b_k\},\{D_k\}}$ is a spectral measure, then $\omega_1$ and $\omega_2$ are also spectral measures.
\end{prop}
\begin{proof}
This is easily obtained from Lemma \ref{lem(3.1)}, Lemma \ref{lem(3.2)} and Proposition \ref{prop(4.10-1)}.
\end{proof}
Based on the above preparations, now we can prove Theorem \ref{th(1.2)}.
\begin{proof}[\bf Proof of  Theorem \ref{th(1.2)}]
Since $\mu_{\{b_k\},\{D_k\}}$ is a spectral measure, it follows from Proposition \ref{prop(3.3)} that $\omega_1$ and $\omega_2$ are also spectral measures. Let $\Lambda$ be a spectrum of $\omega_1$. Then the cardinality of $\Lambda$ is equivalent to the dimension of $L^2(\omega_1)$, that is,
$\#\Lambda=N^{\#\mathbf{A}}$.
According to \eqref{eq(3.8)}, we have
\begin{equation}\label{eq(2.4.1)}
\mathcal{Z}(\hat{\omega}_{1})= \bigcup_{k\in\mathbf{A}} \frac{N^{\mathbf{s}}\mathbf{b}_k(\mathbb{Z}\setminus N\mathbb{Z})}{t'_k} \subset\frac{N^{\mathbf{s}}(\mathbb{Z}\setminus N\mathbb{Z})}{\tilde{t} },
\end{equation}
where $\tilde{t}\in\mathbb{Z}\setminus N\mathbb{Z}$ is the least common multiple of all elements in $\{ t'_k : k\in\mathbf{A}\}$. From \eqref{eq(2.4.1)}, we conclude that $L^2(\omega_{1})$ contains at most $N$ mutually orthogonal
exponential functions. This implies that $\#\Lambda=N^{\#\mathbf{A}}\leq N$, and further we obtain $\#\mathbf{A}=1 $. Hence, there exists  unique $i_1\in\mathbb{N}^+$ such that $ \mathbf{s}_{i_1}=\mathbf{s}$  and $ \mathbf{s}_{j}>\mathbf{s}_{i_1}$ for all $j\in\mathbb{N}^{+}\setminus\{i_1\}$.

Since $\omega_2$ is a spectral measures, we can replace the above $\mu_{\{b_k\},\{D_k\}}$ with $\omega_2$ and repeat the above process. Hence, there exists  unique $i_2\in\mathbb{N}^+\setminus\{i_1\}$ such that  $\mathbf{s}_{j}>\mathbf{s}_{i_2}>\mathbf{s}_{i_1}$ for all $j\in\mathbb{N}^{+}\setminus\{i_1, i_2\}$. Repeat this operations, we can get that  $ \mathbf{s}_i\neq\mathbf{s}_j$ for all $i\neq j\in\mathbb{N}^{+}$. Combining this with  Theorem \ref{th(1.2-1)}, the proof is completed.
\end{proof}
\section{\bf Proof of Theorem \ref{th(1.4)}} \label{sect.4}
 \
In this section, we  focus on proving  $``(ii)\Longrightarrow (i)"$ of Theorem \ref{th(1.4)}.
We first set up some notations in  the rest of this paper and give some lemmas and propositions that are needed in the subsequent proof.

Let  \begin{equation}\label{eq(3.4.4)}
\mathfrak{n}_k:=\max\{j\geq k: \mathbf{s}_k \geq\mathbf{s}_j\}
 \end{equation}
for all $k\geq1$.

\begin{lem}\label{lem(4.1-1)}With some of the above notions, suppose that $\mathbf{s}_i\neq\mathbf{s}_j$ for all $i\neq j$, then $\alpha:=\sup_{k\geq1}\left\{\mathfrak{n}_k-k\right\}<\infty$.
%
\end{lem}
\begin{proof} Since the sequence of positive integers  $\{t_{k}\}_{k=1}^{\infty}$ is bounded, there exist two positive integers $ M$ and $L$ such that $ 1\leq t_k\leq M$ and $0\leq\tau(t_k)\leq L$ far all $k\geq1$. Suppose $\left\{\mathfrak{n}_k-k\right\}_{k=1}^{\infty}$ is unbounded, then there exists an a positive integer $ k_{0}$  such that $\mathfrak{n}_{k_0}-k_0\geq(L+2)( M+2)$.
According to the definition of $\mathbf{s}_{k_0}$ and $\mathfrak{n}_{k_0}$, we have
\begin{equation}\label{eq(3.5.002)}\tau\left(\frac{b_1b_2\cdots b_{\mathfrak{n}_{k_0}}}{Nt_{\mathfrak{n}_{k_0}}}\right)\leq \tau\left(\frac{b_1b_2\cdots b_{k_0} }{Nt_{k_0}}\right).
\end{equation}
Applying \eqref{eq(3.5.002)}, one may get
\begin{equation}\label{eq(3.5.00200)}\begin{aligned}
\tau ( b_{k_{0}+1}b_{k_{0}+2}\cdots b_{\mathfrak{n}_{k_0}})
&=\tau( b_1b_2\cdots b_{\mathfrak{n}_{k_0}} )-\tau(b_1b_2\cdots b_{k_0} )\\&\leq \tau({t_{\mathfrak{n}_{k_0}}})-\tau(t_{k_0})\\&\leq L.
\end{aligned}\end{equation}
For any $i\in\{1, 2,\cdots, \mathfrak{n}_{k_0}-k_0-M-1\}$, write
$$A_i=\{k_0+i, k_0+i+1,\cdots,k_0+i+M+1\}.$$
We claim there exists $i_0\in\{ 1, 2,\cdots, \mathfrak{n}_{k_0}-k_0-M-1\}$ such that $ \Sigma_{j\in A_{i_0}}\tau( b_{j})=0$. Otherwise, for any $i\in T$, there exists $j_i\in A_i$ such that $\tau( b_{j_i})\geq1$, where $T=\{ 1, 2+M+1,3+2(M+1),\cdots, L+1+L(M+1)\}$. Then $\tau ( b_{k_{0}+1}b_{k_{0}+2}\cdots b_{\mathfrak{n}_{k_0}})\geq \Sigma_{i\in T  }\tau( b_{j_i})\geq L+1,$ which  contradicts \eqref{eq(3.5.00200)}. Hence, the claim  follows.

Combining these with $1\leq t_k\leq M$, we conclude that there exist $i_1\neq i_2\in\{i_{0},i_{0}+1,\cdots,i_{0}+M\}$ such that $t_{i_1} = t_{i_2}$ and $\tau(b_1b_2\cdots b_{i_1})=\tau(b_1b_2\cdots b_{i_2})$.
This illustrates that
$$\mathbf{s}_{i_1}=\tau(\frac{b_1b_2\cdots b_{i_1}}{Nt_{i_1}})=\tau( \frac{b_1b_2\cdots b_{i_2}}{Nt_{i_2}})=\mathbf{s}_{i_2},$$
which contradicts the fact $\mathbf{s}_{i_1}\neq\mathbf{s}_{i_2}$.
\end{proof}
\begin{lem}\label{lem(4.1)}
With some of the above notions, suppose that $\mathbf{s}_j>\mathbf{s}_i$,   then following two  statements hold.
\begin{enumerate}[\rm(i).]
 \item    $\mathfrak{n}_j\geq \mathfrak{n}_i$ and $\mathbf{b}_{\mathfrak{n}_i} \mid \mathbf{b}_{\mathfrak{n}_j}$.
\item For any $\lambda_1\in N^{\mathbf{s}_{i}} \mathbf{b}_{\mathfrak{n}_i}(\mathbb{Z}\setminus N\mathbb{Z})$ and $\lambda_2\in N^{\mathbf{s}_{j}} \mathbf{b}_{\mathfrak{n}_j}(\mathbb{Z}\setminus N\mathbb{Z})$, we have $\lambda_1+\lambda_2\in N^{\mathbf{s}_{i}} \mathbf{b}_{\mathfrak{n}_i} (\mathbb{Z}\setminus N\mathbb{Z}).$
 \end{enumerate}
\end{lem}
\begin{proof}
(i). Since $\mathbf{s}_j>\mathbf{s}_i$, we have $\mathbf{s}_j>\mathbf{s}_i\geq\mathbf{s}_{\mathfrak{n}_i}$. Suppose $\mathfrak{n}_i>\mathfrak{n}_j$, then $\mathbf{s}_{\mathfrak{n}_i}>\mathbf{s}_j$, which contradicts the fact $\mathbf{s}_j>\mathbf{s}_{\mathfrak{n}_i}$. Hence, $\mathfrak{n}_j\geq \mathfrak{n}_i$. According to the definition of $\mathbf{b}_{\mathfrak{n}_j}$, we can get $\mathbf{b}_{\mathfrak{n}_j}=\mathbf{b}_{\mathfrak{n}_i}b'_{\mathfrak{n}_{i+1}}\cdots b'_{\mathfrak{n}_j}$, which means $\mathbf{b}_{\mathfrak{n}_i} \mid \mathbf{b}_{\mathfrak{n}_j}$.

(ii). For any $\lambda_1\in N^{\mathbf{s}_{i}} \mathbf{b}_{\mathfrak{n}_i}(\mathbb{Z}\setminus N\mathbb{Z})$ and $\lambda_2\in N^{\mathbf{s}_{j}} \mathbf{b}_{\mathfrak{n}_j}(\mathbb{Z}\setminus N\mathbb{Z})$,
we have $\lambda_1=N^{\mathbf{s}_{i}} \mathbf{b}_{\mathfrak{n}_i}l_{1}$ and $\lambda_2= N^{\mathbf{s}_{j}} \mathbf{b}_{\mathfrak{n}_j}l_{2}$
for some $l_1, l_2\in\mathbb{Z}\setminus N\mathbb{Z}$.
Since $\mathbf{s}_{j}> \mathbf{s}_{i}$,  it follows from (i) that
$$\lambda_1+\lambda_2=  N^{\mathbf{s}_{i}} \mathbf{b}_{\mathfrak{n}_i}l_{1}+N^{\mathbf{s}_{j}} \mathbf{b}_{\mathfrak{n}_j}l_{2}\in N^{\mathbf{s}_{i}} \mathbf{b}_{\mathfrak{n}_i}(\mathbb{Z}\setminus N\mathbb{Z}).$$
Hence, the lemma follows.\end{proof}

We will use the above lemmas to obtain the following Proposition \ref{prop(4.2.1)}. Before we do that, we need to give some important symbolic definitions.
For any $k\in\mathbb{N}^+$,  we write
\begin{equation} \label{eq(4.3)}
\mu_{k}=\delta_{b_{1 }^{-1} D_{1}} \ast  \delta_{ b_{1}^{-1} b_{2}^{-1}   D_{2}}\ast \cdots  \ast \delta_{  b_{1}^{-1} b_{2}^{-1}\cdots b_{k}^{-1}  D_{k}}
\end{equation}
and
\begin{equation} \label{eq(4.2)}
\nu_{>k}=\delta_{b_{k+1 }^{-1} D_{ k+1} } \ast \delta_{ b_{k+1 }^{-1} b_{k+2 }^{-1} D_{k+2}}\ast \delta_{ b_{k+1 }^{-1} b_{k+2 }^{-1} b_{k+3 }^{-1}D_{k+3}} \ast\cdots.
\end{equation}

For any two positive integers $k'>k$, suppose that $\mathbf{s}_i\neq\mathbf{s}_j$ for all $i\neq j\in\mathbb{N}^{+}$,  we define  \begin{equation} \label{eq(4.3-2)} \Lambda_{k, k'}:=\bigcup_{\lambda\in \mathcal{B}_{k, k'}}(\lambda+b_1b_2\cdots b_{m_{k'}}z_{\lambda}),
\end{equation}
where $\Lambda_{k,k'}$ satisfy the following three conditions\\
{\bf (i)} $z_{0}=0$ and $z_{\lambda}\in\mathbb{Z}$;\\
{\bf (ii)} $ m_{k'}\geq\max\{ \mathfrak{n}_{j}: j\leq k', j\in\mathbb{N}^+\}$;\\
{\bf (iii)} $\mathcal{B}_{k,k'}=\bigoplus_{j=k+1}^{k'}\left(N^{\mathbf{s}_j} \mathbf{b}_{\mathfrak{n}_{j}} c_{ j } \{0, 1, \cdots, N-1\}\right)$ with $c_j\in\mathbb{Z}\setminus N\mathbb{Z} $.
\begin{rem}{\rm Under the observation of Lemma \ref{rem(1-1)}, we can obtain that  $\mathcal{B}_{k,k'}$ is a direct sum.}
\end{rem}

\begin{prop}\label{prop(4.2.1)}
 Given a strictly  increasing sequence $\{k_n\}_{n=0}^{\infty}$ with $k_0=0$, let $\Lambda_{k_{n-1},k_{n}}$  be defined by \eqref{eq(4.3-2)}. Suppose that  $\mathbf{s}_i\neq\mathbf{s}_j$ for all $i\neq j$, then
 $$ \Lambda_{n}=\Lambda_{k_0,k_{1}}+\Lambda_{k_{1},k_{2}}+\cdots+\Lambda_{k_{n-1},k_{n}}$$
 is a spectrum of $\mu_{k_{n}}$ and $\Lambda_{n}\subset\Lambda_{n+1}$  for all $n\geq1$, where $\mu_{k_{n}}$ is defined by \eqref{eq(4.3)}.
\end{prop}
\begin{proof}
Obviously, $\Lambda_{n}\subset\Lambda_{n+1}$  for all $n\geq1$.
For any $n\geq1$ and two distinct sequences $ \{ \lambda_j \}_{j=1}^{n}$ and $ \{ \tilde{\lambda}_j \}_{j=1}^{n}$ with $\lambda_j, \tilde{\lambda}_j\in\Lambda_{k_{j-1},k_{j}}$. Let $\lambda= \sum_{j=1}^{n}\lambda_j$ and $\tilde{\lambda}=\sum_{j=1}^{n}\tilde{\lambda}_j$. It follows from \eqref{eq(4.3-2)} that there exist $l_i, \tilde{l}_i\in\{ 0, 1, 2, \cdots, N-1\}$ for $1\leq i\leq k_{n}$ and  $z_{j},\tilde{z}_{j}\in\mathbb{Z}$ for  $1\leq j\leq n$  such that
$$\lambda=\sum_{i=1}^{k_{n}}N^{\mathbf{s}_i} \mathbf{b}_{\mathfrak{n}_{i}}c_i l_i+ \sum_{j=1}^{ n}b_1b_2\cdots b_{m_{k_{j}}}z_{j}\ \ \text{and}\ \ \tilde{\lambda}= \sum_{i=1}^{k_{n}}N^{\mathbf{s}_i} \mathbf{b}_{\mathfrak{n}_{i}}c_i \tilde{l}_i+ \sum_{j=1}^{ n}b_1b_2\cdots b_{m_{k_{j}}}\tilde{z}_{j},$$
 where \begin{align*}  \sum_{i=k_{j-1}+1}^{k_{j}}N^{\mathbf{s}_i} \mathbf{b}_{\mathfrak{n}_{i}}c_i l_i+b_1b_2\cdots b_{k_j}z_j= \lambda_j \ \ \text{and} \ \ \sum_{i=k_{j-1}+1}^{k_{j}}N^{\mathbf{s}_i} \mathbf{b}_{\mathfrak{n}_{i}}c_i \tilde{l}_i+b_1b_2\cdots b_{k_j}\tilde{z}_j=\tilde{\lambda}_j.\end{align*}
 Then
$$\tilde{\lambda}-\lambda= \sum_{i=1}^{k_{n}}N^{\mathbf{s}_i} \mathbf{b}_{\mathfrak{n}_{i}}c_i(\tilde{l}_i - l_i)+ \sum_{j=1}^{ n}b_1b_2\cdots b_{m_{k_{j}}}(\tilde{z}_{j}-z_j)$$
and there exists $i_{0}\in \{ 1,2,\cdots, k_n\}$  such that $$\mathbf{s}_{i_{0}}=\min\left\{ \mathbf{s}_i : N^{\mathbf{s}_i} \mathbf{b}_{\mathfrak{n}_{i}}c_i(\tilde{l}_i - l_i)\neq0, 1\leq i\leq k_{n} \right\}.$$
According to Lemma \ref{lem(4.1)} (ii), we have
$$\tilde{\lambda}-\lambda\in N^{\mathbf{s}_{i_{0}}} \mathbf{b}_{\mathfrak{n}_{i_0}}(\mathbb{Z}\setminus N\mathbb{Z}) + \sum_{j=1}^{ n}b_1b_2\cdots b_{m_{k_{j}}}(\tilde{z}_{j}-z_j).$$
Note that for any $j\in\{1, 2,\cdots, n\}$, the following two statements are easily obtained by  the definition of $\Lambda_{k_{j-1},k_{j}}$  and Lemma \ref{lem(4.1)}(i) :

$(a).$ If $ \tilde{l}_i =l_i$   for all $k_{j-1}+1\leq i\leq k_{j}$, we have  $\tilde{z}_{j}=z_j$;

$(b).$ If $ \tilde{l}_{i_0}\neq l_{i_0}$ for  some $k_{j-1}+1\leq i_0\leq k_{j}$, we have $m_{k_j}\geq\mathfrak{n}_{i_0}$.

This means that $ \tilde{\lambda}-\lambda\in N^{\mathbf{s}_{i_{0}}} \mathbf{b}_{\mathfrak{n}_{i_0}}(\mathbb{Z}\setminus N\mathbb{Z}) + b_1b_2\cdots b_{\mathfrak{n}_{i_0}}\mathbb{Z}$. Based on the definition of $\mathbf{s}_{i_{0}}$ and $\mathfrak{n}_{ i_0}$, it's easy to show that
$$\tilde{\lambda}-\lambda\in N^{\mathbf{s}_{i_{0}}} \mathbf{b}_{\mathfrak{n}_{i_0}}(\mathbb{Z}\setminus N\mathbb{Z})\subset\mathcal{Z}(\hat{\delta}_{b_{1}^{-1}b_{2}^{-1}\cdots b_{i_{0}}^{-1}D_{i_0}}).$$

Therefore, $\tilde{\lambda}-\lambda= \sum_{j=1}^{n}(\tilde{\lambda}_j- \lambda_j)\neq0$ and  $\{ \tilde{\lambda}, \lambda\}$ is an orthogonal set of  $\mu_{k_{n}}$.
According to the arbitrariness of two distinct sequences $ \{ \lambda_j \}_{j=1}^{n}$ and $ \{ \tilde{\lambda}_j \}_{j=1}^{n}$ with $\lambda_j, \tilde{\lambda}_j\in\Lambda_{k_{j-1},k_{j}}$, we have $\Lambda_{n}=\Lambda_{k_0,k_{1}}\oplus\Lambda_{k_{1},k_{2}}
\oplus\cdots\oplus\Lambda_{k_{n-1},k_{n}}$, i.e., $\#\Lambda_{n}=N^{k_{n}}$ is equivalent to the dimension of
$L^2( \mu_{k_{n}})$, and $\Lambda_{n}$ is an orthogonal set of  $\mu_{k_{n}}$. Then $\Lambda_{n}$ is a spectrum of $\mu_{k_{n}}$
 and the proof is complete.
\end{proof}

The following well-known result will be useful in this section.
\begin{prop}\cite[Lemma 2.2]{AFL19}\label{prop(4.1)}
Let $\{\nu_k\}_{k=1}^{\infty} $ be a sequence
of probability measures with compact support set. Then $\{ \hat{\nu}_{k} \}_{k=1}^{\infty} $ is equicontinuous.
\end{prop}
We will give the definition of equi-positive family, which helps us to understand the proof of Proposition \ref{prop(4.2)}. And the following  Proposition \ref{prop(4.2)} plays an important role in studying the sufficiency of Theorem \ref{th(1.4)}.
\begin{defi}\label{de(4.3)}
Let $\Xi$ be a collection of probability measures on compact set $[0, 1]$. We
say that $\Xi$ is an equi-positive family if there exists $ \varepsilon_0> 0$ such that for all $\nu \in\Xi$ and $x \in[0, 1] $, there exists an integer $\mathbf{k}_{ \nu, x}$ such that
$|\hat{\nu}(x+\mathbf{k}_{\nu, x })|\geq  \varepsilon_0$.
\end{defi}
\begin{prop}\label{prop(4.2)}
Let $\mu_{\{b_k\},\{D_k\}}$ and $\nu_{>k}$  be defined by \eqref{eq(2.08)} and \eqref{eq(4.2)} respectively, and let integer $m_0\geq1$. Suppose $b_k>(N-1)t_k$ for all $k\geq m_{0}$, then there exist $C>0$ and $\theta_0>0$ such that for any $x\in[0,1]$ and  $k\geq m_{0}$, there exists an integer $\mathbf{k}_{k, x}$ such that
$$
|\hat{\nu}_{>k}(x+y+\mathbf{k}_{k, x})|>C
$$
for any $y\in[-\theta_0, \theta_0]$, where $\mathbf{k}_{k, 0}=0$  for all $k\geq m_{0}$.
\end{prop}
\begin{proof}
Since the Cantor-Moran measure $\nu_{>k}$  is supported on a compact set
\begin{equation}\label{eq(0.0)}
\left\{\sum_{n=1}^{\infty} \frac{ d_{k+n} }{b_{k+1} b_{k+2} \cdots b_{k+n}} : d_{k+n}\in D_{k+n}\ \ \text{for all}\ \ n\geq1\right\},
 \end{equation}
we have  the support  $\text{spt}(\nu_{>k})\subset\left[0, \sum_{n=1}^{\infty} \frac{ (N-1)t_{k+n} }{b_{k+1} b_{k+2} \cdots b_{k+n}}\right].$
As $b_k>(N-1)t_k$ for all $k\geq m_0$, one has
$$\begin{aligned}
\sum_{n=1}^{\infty} \frac{ (N-1)t_{k+n} }{b_{k+1} b_{k+2} \cdots b_{k+n}}&\leq
\sum_{n=1}^{\infty} \frac{ [\left(N-1\right)t_{k+n}+1]-1 }{[\left(N-1\right)t_{k+1} +1][\left(N-1\right)t_{k+2} +1]\cdots [\left(N-1\right)t_{k+n}+1]}
\\&\leq1
\end{aligned}$$
for all $k\geq m_{0}$. Hence, $\text{spt}(\nu_{>k})\subset[0, 1]$ for all $k\geq m_{0}$. Since $\#D_k=N $ for all $k\geq1$, it follows from \cite[Theorem 5.4]{AFL19} that $\{ \nu_{>k}\}_{k=m_0}^{\infty}$ is an equi-positive family. Hence there exists $C>0$ such that for any $x\in[0,1]$ and  $k\geq m_{0}$, there exists an integer $\mathbf{k}_{k, x}$ with $\mathbf{k}_{k, 0}=0$ such that
$
|\hat{\nu}_{>k}(x+\mathbf{k}_{k, x})|>2C$.
 In view of Proposition $\ref{prop(4.1)}$, we have  $\{ \hat{\nu}_{>k} \}_{k=m_0}^{\infty} $ is equicontinuous. Thus there exists  $\theta_0>0$ such that
$
|\hat{\nu}_{>k}(x+y+\mathbf{k}_{k, x})|>C
$
for any  $y\in[-\theta_0, \theta_0]$, and the proposition follows.
\end{proof}
  Theorem 2.3 in \cite {AHH19} give a discriminating method that $\Lambda$ become a spectrum of the Cantor-Moran measure $\mu_{\{b_k\},\{D_k\}}$ . For the convenience of the readers, we provide its proof process and improve it to make it simpler to use.

\begin{thm}\label{prop(4.2.2)}
 With the above notations, let $\{k_n\}_{n=1}^{\infty}$ be a strictly increasing sequence, and let $\Lambda_{n}$ be a spectrum of $\mu_{k_{n}}$ for all $n\geq1$. If $\Lambda_{n}\subset\Lambda_{n+1}$  for all $n\geq1$ and  there exists $\varepsilon_0>0$   such that for any $n\geq1$
$$ \left|\hat{\nu}_{>k_{n}}\left( \frac{ \lambda}{b_1b_2 \cdots b_{k_{n}}}\right)\right|\geq\varepsilon_0$$
for all  $\lambda\in\Lambda_{n}$, then $\Lambda=\bigcup_{n=1}^{\infty}\Lambda_{n}$ is a spectrum of $\mu_{\{b_k\},\{D_k\}}$.
\end{thm}
\begin{proof}
According to Proposition \ref{prop(4.1)}, there exists $\rho_{0}>0$ such that for any $n\geq1$,
$$ \left|\hat{\nu}_{>k_{n}}\left( \frac{\xi+\lambda}{b_1b_2 \cdots b_{k_{n}}}\right)\right|\geq\frac{\varepsilon_0}{2}$$
for all  $\lambda\in\Lambda_{n}$ and $\xi\in[-\rho_0, \rho_0]$.
Let $Q_{k_n}(\xi)= \sum_{\lambda \in \Lambda_{n}}|\hat{\mu}(\xi+\lambda)|^{2}$ for $\xi\in[-\rho_0, \rho_0]$. Then
\begin{equation}\label{eq(0.001)}
Q_{\Lambda}(\xi)= \sum_{\lambda \in \Lambda }|\hat{\mu}(\xi+\lambda)|^{2}=  \lim_{n\rightarrow\infty}Q_{k_n}(\xi).
\end{equation}
For any $p\geq1$, we have
\begin{align}\label{eq(0.002)}
Q_{k_{n+p}}(\xi)\nonumber&=Q_{k_n}(\xi)+ \sum_{\lambda \in \Lambda_{n+p}\setminus \Lambda_{n} }\left|\hat{\mu}(\xi+\lambda)\right|^{2}
\nonumber\\&=Q_{k_n}(\xi)+ \sum_{\lambda \in \Lambda_{n+p}\setminus \Lambda_{n} }\left|\hat{\mu}_{k_{n+p}}\left(\xi+\lambda\right)\right|^{2}\left|\hat{\nu}_{>k_{n+p}}
\left(  b_{1}^{-1}b_{2}^{-1} \cdots b_{k_{n+p}}^{-1} (\xi+\lambda)\right)\right|^{2}
\nonumber\\&\geq Q_{k_n}(\xi)+\frac{\varepsilon_{0}^{2}}{4}\sum_{\lambda \in \Lambda_{n+p}\setminus \Lambda_{n} }\left|\hat{\mu}_{k_{n+p}}\left(\xi+\lambda\right)\right|^{2}
\nonumber\\&= Q_{k_n}(\xi)+\frac{\varepsilon_{0}^{2}}{4} \left(1-\sum_{\lambda \in \Lambda_{n} }\left|\hat{\mu}_{k_{n+p}}\left(\xi+\lambda\right)\right|^{2}\right).
\end{align}
 Letting $p\rightarrow\infty$,  it follows from \eqref{eq(0.001)} and  \eqref{eq(0.002)} that
$Q_{\Lambda}(\xi)- Q_{k_n}(\xi)\geq\frac{\varepsilon_{0}^{2}}{4}\left(1-Q_{k_n}(\xi)\right).$
Taking $n\rightarrow\infty$, we have $Q_{\Lambda}(\xi)=\sum_{\lambda \in \Lambda }|\hat{\mu}(\xi+\lambda)|^{2}=1$ for $\xi\in[-\rho_0, \rho_0]$.
In view of Proposition \ref{prop(2.3)}, we have $\Lambda=\bigcup_{n=1}^{\infty}\Lambda_{n}$ is a spectrum of $\mu_{\{b_k\},\{D_k\}}.$
\end{proof}

Next, we will decompose the proof of  $``(ii)\Longrightarrow (i)"$ of Theorem \ref{th(1.4)} into  the following two cases.

 \textbf{Case I:}
 There exists an infinite subsequence $\{k_n\}_{n=1}^{\infty}$ of $\mathbb{N}^+$
 such that $\min\{\mathbf{s}_j : j> k_n\}> \max\{\mathbf{s}_j : j\leq k_n\}$ for all $n\geq1$.

 \textbf{Case II:} There exists $k_0\in\mathbb{N}^+$ such that $\min\{\mathbf{s}_j : j> k\}< \max\{\mathbf{s}_j : j\leq k\}$ for all $k\geq k_0$.

\subsection{\bf Case  I }
In Case I, the proof of Theorem \ref{th(1.4)}  is relatively simple, and we can use the above preparation to prove it directly.
\begin{thm}\label{th(1.4_1)}
 Under the assumption of Theorem \ref{th(1.4)}, suppose that  $\mathbf{s}_i\neq\mathbf{s}_j$ for all $i\neq j$ and
there exists a subsequence $\{k_{n}\}_{n=1}^{\infty}$ of $\{k\}_{k=1}^{\infty}$ such that $\min\{\mathbf{s}_j: j> k_n\}> \max\{\mathbf{s}_j: j\leq k_n\}$ for all $n\geq1$,
then $\mu_{\{b_k\},\{D_k\}}$ is a spectral measure.
\end{thm}
\begin{proof}Recall that $m_0$ is defined in Theorem \ref{th(1.4)}.
Let $k_{n_1}\geq m_0$ and $k_{n_1}\in\{k_{n}\}_{n=1}^{\infty}$, and let $$\mathcal{B}_{0,k_{n_1}}=\bigoplus_{j=1}^{k_{n_1}}\left(N^{\mathbf{s}_j} \mathbf{b}_{\mathfrak{n}_{j}}\{0, 1, \cdots, N-1\}\right).$$
 According to  Proposition \ref{prop(4.2)},  for any $\lambda_1\in\mathcal{B}_{0,k_{n_1}}$ there exists  an integer $\mathbf{k}_{1, \lambda_1}$ such that
$$
\left|\hat{\nu}_{>k_{n_1}}( \frac{\lambda_1}{b_1b_2\cdots b_{k_{n_1}}}+\mathbf{k}_{1, \lambda_1})\right|>C
$$
for some $C>0$, where $\mathbf{k}_{1,0}=0$.  Let
$$\Lambda_{n_1}:=\Lambda_{0,k_{n_1}}=\bigcup_{\lambda_1\in \mathcal{B}_{0,k_{n_1}}}(\lambda_1+b_1b_2\cdots b_{k_{n_1}} \mathbf{k}_{1, \lambda_1} ).$$ Then
 $|\hat{\nu}_{>k_{n_1}}( \frac{\lambda}{b_1b_2\cdots b_{k_{n_1}}})|>C$
for any $ \lambda\in\Lambda_{n_1}$. Since $\min\{\mathbf{s}_j : j> k_{n_1}\}> \max\{\mathbf{s}_j : j\leq k_{n_1}\}$, we have $ {k_{n_1}}\geq\max\{ \mathfrak{n}_{j}: j\leq k_{n_1}\}$. This means that $\Lambda_{0,k_{n_1}}$ satisfies  $(i)-(iii)$ of \eqref{eq(4.3-2)}. It follows from  Proposition \ref{prop(4.2.1)} that $\Lambda_{n_1}$ is a spectrum of $\mu_{k_{n_1}}$.

Let $k_{n_2}\in\{k_{n}\}_{n=1}^{\infty}$ satisfy $k_{n_2}>k_{n_1}$ and
$( b_1b_2\cdots b_{k_{n_2}})^{-1}\Lambda_{n_1}\subset[ -\theta_0, \theta_0]$, where $\theta_0$ is given in Proposition \ref{prop(4.2)}. Define  $$\mathcal{B}_{k_{n_1},k_{n_2}}=\bigoplus_{j=k_{n_1}+1 }^{k_{n_2}}\left(N^{\mathbf{s}_j} \mathbf{b}_{\mathfrak{n}_{j}}\{0, 1, 2, \cdots, N-1\}\right).$$
According to  Proposition \ref{prop(4.2)},  for any $\lambda_1\in\Lambda_{n_1}$ and $\lambda_2\in\mathcal{B}_{k_{n_1},k_{n_2}}$ there exists  an integer $\mathbf{k}_{2,\lambda_2}$ such that
$$
|\hat{\nu}_{>k_{n_2}}( \frac{\lambda_1+\lambda_2}{b_1b_2\cdots b_{k_{n_2}}}+\mathbf{k}_{2, \lambda_2})|>C
$$
and $\mathbf{k}_{2, 0}=0$. Let $ \Lambda_{k_{n_1}, k_{n_2}}=\cup_{\lambda_2\in \mathcal{B}_{k_{n_1}, k_{n_2}}} (\lambda_2+b_1b_2\cdots b_{k_{n_2}} \mathbf{k}_{2, \lambda_2})$ and  $ \Lambda_{n_2}= \Lambda_{0,k_{n_1}}+ \Lambda_{k_{n_1}, k_{n_2}}$.
Then
$|\hat{\nu}_{>k_{n_2}}( \frac{\lambda}{b_1b_2\cdots b_{k_{n_2}}})|>C$
for any $ \lambda\in\Lambda_{n_2}$.
Since $\min\{\mathbf{s}_j : j> k_{n_2}\}> \max\{\mathbf{s}_j : j\leq k_{n_2}\}$, we have   $ {k_{n_2}}\geq\max\{ \mathfrak{n}_{j}: j\leq k_{n_2}\}$ and   $\Lambda_{k_{n_1}, k_{n_2}}$ satisfies  $(i)-(iii)$ of \eqref{eq(4.3-2)}. It follows from Proposition \ref{prop(4.2.1)} that $\Lambda_{n_2}$ is a spectrum of $\mu_{k_{n_2}}$ and $ \Lambda_{n_1}\subset\Lambda_{n_2}$.

Repeat this operation,  we can find a strictly  increasing sequence $\{k_{n_i}\}_{i=1}^{\infty}$ such that for any $i\geq1$, the following three statements hold: (i) $\Lambda_{n_i}\subset \Lambda_{n_i+1}$; (ii) $\Lambda_{n_i}$ is a spectrum of $\mu_{k_{n_i}}$; (iii)
$|\hat{\nu}_{>k_{n_i}}( \frac{\lambda}{b_1b_2\cdots b_{k_{n_i}}})|>C$ for any $ \lambda\in\Lambda_{n_i}$.
Combining  these with Theorem \ref{prop(4.2.2)}, we have
$\Lambda=\bigcup_{i=1}^{\infty}\Lambda_{n_i}$ is a spectrum of $\mu_{\{b_k\},\{D_k\}}.$ Thus the proof follows.
\end{proof}
\subsection{\bf Case II }
%
To prove  Case II, we need to make some technical preparations, that is, construct the appropriate $\Lambda=\bigcup_{n=1}^{\infty}\Lambda_{n}$  to satisfy the conditions of Theorem \ref{prop(4.2.2)}, but the construction method is different from Case I. The $m_0$ mentioned in this section comes from Theorem \ref{th(1.4)} and will not be hinted at later for simplicity of writing. We begin with some propositions.
\begin{prop}\label{lem(4.4-1)}
  Under the assumption of Theorem \ref{th(1.4)}, suppose that $\mathbf{s}_i\neq\mathbf{s}_j$ for all $i\neq j$ and there exists  $k_0 \geq m_0$ such that $\min\{\mathbf{s}_j : j> k\}< \max\{\mathbf{s}_j : j\leq k\}$ for any $k\geq k_0$, then the following two statements hold.
\begin{enumerate}[\rm(i).]
\item There exists a positive integer $\beta$ such that for any $k\geq k_0$,
$$
\max\{t_n': n\geq1\}\mathbf{b}_{\mathfrak{n}_{k}}< \mathbf{b}_{i}
$$ if $ i \geq k+\beta $.
\item $\tau(b_k)\leq \max\{\tau(t_n): n\geq1\}$ for any $k\geq k_0 +1$.
 \end{enumerate}
\end{prop}
\begin{proof}

(i).  We first prove the following claim.

{\rm {\bf Claim 2.} For any $k\geq k_0$, we have
$\#\left\{i: \frac{\tau(b_{k+i})}{\tau(t_{k+i})}\leq 1, 0\leq i\leq \alpha\right\}\geq1$, where $\alpha$ is defined in Lemma \ref{lem(4.1-1)}.}
\begin{proof}[\bf Proof of  Claim 2]
Suppose that
$\#\left\{i: \frac{\tau(b_{ \bar{k} +i})}{\tau(t_{ \bar{k} +i})}\leq 1, 0\leq i\leq\alpha\right\}=0$ for some $\bar{k}\geq k_0$, then $ \tau(b_{\bar{k}+i})- \tau(t_{\bar{k}+i})>0$ for all $0\leq i\leq\alpha$. This means $\mathbf{s}_{\bar{k} }<\mathbf{s}_{\bar{k}+1}<\cdots<\mathbf{s}_{\bar{k}+\alpha}.$ According to Lemma \ref{lem(4.1-1)}, we have $ \bar{k}=\mathfrak{n}_{\bar{k}}$. Combining these with $\min\{\mathbf{s}_j : j> \bar{k}\}< \max\{\mathbf{s}_j : j\leq \bar{k}\}$, there exists $i_{\bar{k}}<\bar{k}$ such that 
$\mathbf{s}_{\bar{k}}<\mathbf{s}_{i_{\bar{k}}}$. Note that $ \mathbf{s}_{\bar{k}}=\tau(b_{1}b_{2}\cdots b_{\bar{k}})-\tau(t_{\bar{k}})-1>\tau(b_{1}b_{2}\cdots b_{\bar{k}-1})-1\geq\mathbf{s}_{i_{\bar{k}}},$ which contradicts $ \mathbf{s}_{\bar{k}}<\mathbf{s}_{i_{\bar{k}}}$.  The claim follows.
\end{proof}
According to Claim 2, for any  $k\geq  k_0 $ there exists $0\leq i_0 \leq\alpha$ such that $ \tau(b_{k+i_0}) \leq \tau(t_{k+i_0})$. We have $b_{k+i_0}'>1$ since $b_{k+i_0}>(N-1)t_{k+i_0} $. This imply that at least one of  $b_{k}', b_{k+1}',\cdots, b_{k+\alpha}'$ is greater than or equal to two. Since  $\{t_{k}\}_{k=1}^{\infty}$ is bounded, there exists a positive integer $\gamma$ such that $\max\{t_n': n\geq1\}<2^{\gamma}.$ Hence,
 $ \max\{t_n': n\geq1\}\mathbf{b}_{\mathfrak{n}_{k}}< \mathbf{b}_{\mathfrak{n}_{k}+\gamma(\alpha+1)}
$ for all $k\geq k_0 $. Making $\beta=\mathfrak{n}_{k}+\gamma(\alpha+1)-k$, we have  $\max\{t_n': n\geq1\}\mathbf{b}_{\mathfrak{n}_{k}}<\mathbf{b}_{i}$
if $i\geq k+ \beta $.

(ii). Suppose that there exists $\tilde{k} \geq k_0+1$ such that $\tau(b_{\tilde{k} })> \max\{\tau(t_n): n\geq1\}$. For any $k_1$ and  $k_2$ satisfying $k_1<\tilde{k}\leq k_2$, we have
$ \mathbf{s}_{k_2} =\tau(b_{1}b_{2}\cdots b_{k_2})-\tau(t_{k_2})-1\geq \tau(b_{1}b_{2}\cdots b_{\tilde{k}})-\tau(t_{k_2})-1>\tau(b_{1}b_{2}\cdots b_{k_{1}})-1 \geq \mathbf{s}_{k_1}$. From the arbitrariness of $k_1$  and $k_2$, we get that  $\min\{\mathbf{s}_j : j> \tilde{k}-1\}>\max\{\mathbf{s}_j : j\leq \tilde{k}-1\}$, which contradicts  our assumption.
\end{proof}
\begin{prop}\label{prop(4.4.2)}
 Under the assumption of Theorem \ref{th(1.4)}, suppose that  $\mathbf{s}_i\neq\mathbf{s}_j$ for all $i\neq j$ and there exists $k_0\geq m_0$ such that $\min\{\mathbf{s}_j : j> k\}< \max\{\mathbf{s}_j : j\leq k\}$ for any $k\geq k_0 $. Then there exist  $\epsilon_0,\vartheta_0>0$ such that for any $ k_2>k_1\geq k_0$, we can choose a appropriate $\mathcal{B}_{k_1,k_2}:=\bigoplus_{j=k_1+1}^{k_{2}}\left(N^{\mathbf{s}_j} \mathbf{b}_{\mathfrak{n}_{j}}c_{ j } \{0, 1, 2, \cdots, N-1\}\right)$ with $c_j\in\mathbb{Z}\setminus N\mathbb{Z}$ to make
 \begin{equation} \label{eq(4.3-3)}\prod_{i=1}^{\alpha} \left| \hat{\delta}_{b_{k_2 +1}^{-1}b_{k_2 +2}^{-1} \cdots b_{k_2 +i}^{-1}D_{k_{2}+i}}\left(\xi+\frac{\lambda}{b_{1}b_{2 } \cdots b_{k_2} }\right) \right|>\epsilon_0
\end{equation}
for any $ \xi\in[-\vartheta_0, \vartheta_{0}]$ and $\lambda\in\mathcal{B}_{k_1,k_2}$, where $\alpha$ is defined in Lemma \ref{lem(4.1-1)}.
\end{prop}
\begin{proof}
Let $$\Omega_1=\left\{j: \max\{t_{k_2+i}': 1\leq i\leq\alpha\}\mathbf{b}_{\mathfrak{n}_{j}}< \mathbf{b}_{k_2+1},  k_1+1\leq j\leq k_2 \right\}$$
and
$$\Omega_2=\left\{j: \max\{t_{k_2+i}': 1\leq i\leq\alpha\}\mathbf{b}_{\mathfrak{n}_{j}}\geq \mathbf{b}_{k_2+1},  k_1+1\leq j\leq k_2 \right\}.$$
This imply that $\Omega_1\cap\Omega_2=\emptyset$ and $\mathfrak{n}_{j}<k_2+1$ for all $j\in\Omega_1$. We choose
\begin{equation*} \label{eq(4.3-3-1-1)}
\mathcal{B}_{k_1,k_2}=\bigoplus_{j=k_1+1}^{k_{2}}\left(N^{\mathbf{s}_j} \mathbf{b}_{\mathfrak{n}_{j}}c_{ j } \{0, 1, 2, \cdots, N-1\}\right)
\end{equation*}
satisfying
$$
c_j =
 \left\{
\begin{array}{ll}
 (-1)^{\mathbf{s}_j} , &\text{if}\ \ j\in\Omega_1; \\
 b'_{\mathfrak{n}_{j}+1} b'_{\mathfrak{n}_{j}+2}\cdots b'_{k_{2}+\alpha}, &\text{if}\ \ j\in\Omega_2.
\end{array}
\right.
$$
Next, we prove the following claim.

{\rm {\bf Claim 3.} There exists  $\tilde{\epsilon }>0$  such that
$$\left| \hat{\delta}_{\{0,1, \cdots, N-1\}}\left(\frac{\lambda t_{k_{2+i}}}{b_{1}b_{2 } \cdots b_{k_{2+i}} }\right) \right| >  \tilde{\epsilon} $$
for any $i\in\{1,2, \cdots, \alpha\}$ and  $\lambda\in\mathcal{B}_{k_1,k_2}$.}
\begin{proof}[\bf Proof of  Claim 3]
For any $i\in\{1,2, \cdots, \alpha\}$ and $\lambda\in\mathcal{B}_{k_1,k_2}$, there exist $l_j\in\{0, 1, 2, \cdots, N-1\}$ for $k_1+1\leq j\leq k_2$ such that $$\lambda =\sum_{j\in\Omega_1\cup\Omega_2} N^{\mathbf{s}_j} \mathbf{b}_{\mathfrak{n}_{j}} c_{ j } l_j=\sum_{j\in\Omega_1 } N^{\mathbf{s}_j} \mathbf{b}_{\mathfrak{n}_{j}} c_{ j } l_j+\sum_{j\in \Omega_2} N^{\mathbf{s}_j} \mathbf{b}_{\mathfrak{n}_{j}} c_{ j } l_j.$$
If $\{j: l_j\neq0,  j\in \Omega_1\}\neq\emptyset$, take $j_1\in\Omega_1$ such that
$$\mathbf{s}_{j_1}=\max\{\mathbf{s}_{j}: l_j\neq0,  j\in \Omega_1\}.$$
Similarly, if $\{j: l_j\neq0, j\in \Omega_2\}\neq\emptyset$, take $j_2\in\Omega_2$ such that $$\mathbf{s}_{j_2}=\min\{\mathbf{s}_{j}: l_j\neq0,  j\in \Omega_2\}.$$
In fact, $\mathbf{s}_{j_1}<\mathbf{s}_{j_2}$. Otherwise, it follows from Lemma \ref{lem(4.1)} (i) that $\mathbf{b}_{\mathfrak{n}_{j_2}} \mid \mathbf{b}_{\mathfrak{n}_{j_1}}$, which
contradicts the definition of $\Omega_1$ and $\Omega_2$.
Hence, we have $\mathbf{s}_{j_1}<\mathbf{s}_{j_2}$. In the following, we will make a classified discussion according to the situation of $\Omega_1, \Omega_2$ and $ l_j $.

(I). For any $m\in\{1,2\}$, if $\Omega_m=\emptyset$ or $l_j=0$ for all $j\in\Omega_m$, we have $\sum_{j\in\Omega_m}\frac{N^{\mathbf{s}_j} \mathbf{b}_{\mathfrak{n}_{j}}c_j l_j t_{k_2+i}}{b_{1}b_{2 } \cdots b_{k_2+i}}=0$.

(II). If $\Omega_1\neq\emptyset$ and  $l_j\neq0$ for some $j\in \Omega_1 $. For any $j\in\Omega_1$, according to the definition of $\Omega_1$, we have $t_{k_2+i}'<b_{\mathfrak{n}_{j+1}}'b_{\mathfrak{n}_{j+2}}'\cdots b_{ k_2+1 }'$, and we claim $ \mathbf{s}_{j_1}<\mathbf{s}_{k_2+i}$ for $1\leq i\leq\alpha$. Otherwise, by Lemma \ref{lem(4.1)} (i) and the definition of $\Omega_1$, we get $\mathbf{b}_{\mathfrak{n}_{k_2}+i}|\mathbf{b}_{\mathfrak{n}_{j_1}}$ and
$\mathbf{b}_{\mathfrak{n}_{j_1}}
\geq\mathbf{b}_{\mathfrak{n}_{k_2}+i}
\geq\mathbf{b}_{\mathfrak{n}_{k_2}+1}
>\mathbf{b}_{\mathfrak{n}_{j_1}},$
which is a contradiction and the claim follows.
Since $\mathbf{s}_{j_1}=\max\{\mathbf{s}_{j}: l_j\neq0, j\in \Omega_1\}$ and $\mathbf{s}_{j}\neq\mathbf{s}_{i}$ for $i\neq j$, we have
\begin{align}
\left|\sum_{j\in\Omega_1 }\frac{N^{\mathbf{s}_j} \mathbf{b}_{\mathfrak{n}_{j}}c_j l_j t_{k_2+i}}{b_{1}b_{2 } \cdots b_{k_2+i}}\right|\nonumber&=\left|\sum_{j\in\Omega_1 }\frac{ (-1)^{\mathbf{s}_j} N^{\mathbf{s}_j-\mathbf{s}_{k_2+i}-1} l_j t_{k_2+i}'}{b_{\mathfrak{n}_{j+1}}'b_{\mathfrak{n}_{j+2}}'\cdots b_{ k_2+i }'}\right|=\left|\sum_{j\in\Omega_1 }\frac{(-1)^{\mathbf{s}_{j_1}-s_j} N^{\mathbf{s}_{j_1}-\mathbf{s}_{k_2+i}-1} l_j }{N^{s_{j_1}-s_j}}\frac{t_{k_2+i}'}{b_{\mathfrak{n}_{j}+1}'b_{\mathfrak{n}_{j}+2}'\cdots b_{k_2+i }'}\right|\end{align}
Since $j\in\Omega_1$, the definition of $\Omega_1$ shows $\frac{t_{k_2+i}'}{b_{\mathfrak{n}_{j}+1}'b_{\mathfrak{n}_{j}+2}'\cdots b_{k_2+i }'}<1$. Hence,
\begin{align}\label{eq(0.002000)}
\left|\sum_{j\in\Omega_1 }\frac{N^{\mathbf{s}_j} \mathbf{b}_{\mathfrak{n}_{j}}c_j l_j t_{k_2+i}}{b_{1}b_{2 } \cdots b_{k_2+i}}\right|\nonumber &<N^{\mathbf{s}_{{j_1}}-\mathbf{s}_{k_2+i}-1}\left| \sum_{j\in\Omega_1,\mathbf{s}_{j_1}-\mathbf{s}_j\in 2 \mathbb{Z} }\frac{  l_j}{N^{\mathbf{s}_{j_1}-\mathbf{s}_j}}-\sum_{j\in\Omega_1,\mathbf{s}_{j_1}-\mathbf{s}_j\in 2 \mathbb{Z}+1}\frac{  l_j}{N^{\mathbf{s}_{j_1}-\mathbf{s}_j}} \right|
\nonumber\\&\leq N^{\mathbf{s}_{{j_1}}-\mathbf{s}_{k_2+i}-1}(N-1)
\max\left\{\sum_{j\in\Omega_1,\mathbf{s}_{j_1}-\mathbf{s}_j\in 2 \mathbb{Z} }\frac{1}{N^{\mathbf{s}_{j_1}-\mathbf{s}_j}},
\sum_{j\in\Omega_1,\mathbf{s}_{j_1}-\mathbf{s}_j\in 2 \mathbb{Z}+1}\frac{  1}{N^{\mathbf{s}_{j_1}-\mathbf{s}_j}}\right\}
\nonumber \\&=\frac{N^{\mathbf{s}_{{j_1}}+1-\mathbf{s}_{k_2+i}}}{N+1}.
\end{align}
We have
$\left|\sum_{j\in\Omega_1}\frac{ N^{\mathbf{s}_j} \mathbf{b}_{\mathfrak{n}_{j}}c_j l_j t_{k_2+i}}{b_{1}b_{2 } \cdots b_{k_2+i}}\right|<  \frac{1}{N+1} $,
which means $\min\left\{\left|\sum_{j\in\Omega_1}\frac{ N^{\mathbf{s}_j} \mathbf{b}_{\mathfrak{n}_{j}}c_j l_j t_{k_2+i}}{b_{1}b_{2 } \cdots b_{k_2+i}}-l\right|:l\in
\mathcal{Z}(\hat{\delta}_{\{0,1,\cdots,N-1\}})\right\}\geq\frac{1}{N(N+1)}$. Therefore, there exists $\epsilon_{i_0}>0$ such that
\begin{align}\label{eq(4.3-3-1-1--4)}\left| \hat{\delta}_{\{0,1, \cdots, N-1\}}\left(\frac{\sum_{j\in\Omega_1} N^{\mathbf{s}_j} \mathbf{b}_{\mathfrak{n}_{j}}c_j l_j t_{k_2+i}}{b_{1}b_{2 } \cdots b_{k_2+i}}+z\right) \right| >\epsilon_{i_0}
\end{align}
for any $z\in\mathbb{Z}$.

(III). If $\Omega_2\neq\emptyset$ and $l_j\neq0$ for some $j\in \Omega_2$. For any $j\in\Omega_2$, we have
\begin{align}\label{eq(4.3-3-1-1--3)}
 \sum_{j\in\Omega_2}\frac{ N^{\mathbf{s}_j} \mathbf{b}_{\mathfrak{n}_{j}}c_j l_j t_{k_2+i}}{b_{1}b_{2 } \cdots b_{k_2+i}} = \sum_{j\in\Omega_2}\frac{  N^{\mathbf{s}_j-\mathbf{s}_{k_2+i}-1}  \mathbf{b}_{k_{2}+\alpha}l_j t_{k_2+i}'}{\mathbf{b}_{k_2+i}}\in \frac{\mathbb{Z}\setminus N\mathbb{Z} }{ N^{\mathbf{s}_{k_2+i}+1-\mathbf{s}_{j_2}}}.
\end{align}

(i). If $\mathbf{s}_{j_2}>\mathbf{s}_{k_2+i}$, then $\sum_{j\in\Omega_2}\frac{ N^{\mathbf{s}_j} \mathbf{b}_{\mathfrak{n}_{j}}c_j l_j t_{k_2+i}}{b_{1}b_{2 } \cdots b_{k_2+i}}  \in \mathbb{Z}.$

(ii). If $\mathbf{s}_{j_2}<\mathbf{s}_{k_2+i}$. Writing $\omega=\mathbf{s}_{k_2+i}+1-\mathbf{s}_{j_2}$, then $\omega\geq2$. According to (I) and \eqref{eq(0.002000)}, we have
$$\left|\sum_{j\in\Omega_1}\frac{ N^{\mathbf{s}_j} \mathbf{b}_{\mathfrak{n}_{j}}c_j l_j t_{k_2+i}}{b_{1}b_{2 } \cdots b_{k_2+i}}\right|<\frac{N^{\mathbf{s}_{{j_1}}+1-\mathbf{s}_{k_2+i}}}{N+1}
\leq\frac{N^{\mathbf{s}_{{j_2}}-\mathbf{s}_{k_2+i}}}{N+1}
=\frac{N}{N^{\omega }(N+1)}.$$  By \eqref{eq(4.3-3-1-1--3)}, there exist  $z_0\in\mathbb{Z}$, $a_1\in\{1,2,\cdots,N-1\}$ and $a_m\in\{0,1,2,\cdots,N-1\}$ with $2\leq m \leq\omega$ such that $$\sum_{j\in\Omega_2}\frac{ N^{\mathbf{s}_j} \mathbf{b}_{\mathfrak{n}_{j}}c_j l_j t_{k_2+i}}{b_{1}b_{2 } \cdots b_{k_2+i}}=z_0+ \frac{a_{\omega}}{N}+\frac{a_1+a_2N+a_3N^2+\cdots+a_{\omega-1}N^{\omega-2}}{ N^{\omega}}.$$
By some simple calculations, we have
\begin{equation} \label{eq(4.3-3-1-1--2)}
z_0+ \frac{a_{\omega}}{N}+\frac{1}{ N^{\omega}}  \leq \sum_{j\in\Omega_2}\frac{ N^{\mathbf{s}_j} \mathbf{b}_{\mathfrak{n}_{j}}c_j l_j t_{k_2+i}}{b_{1}b_{2 } \cdots b_{k_2+i}} <z_0+ \frac{a_{\omega}}{N}+ \frac{N^{\omega-1}-1}{ N^{\omega}}.
\end{equation}
According to (I) and (II), we obtain $|\sum_{j\in\Omega_1}\frac{ N^{\mathbf{s}_j} \mathbf{b}_{\mathfrak{n}_{j}}c_j l_j t_{k_2+i}}{b_{1}b_{2 } \cdots b_{k_2+i}}|<\frac{N}{N^{\omega}(N+1)}$, which shows \begin{align*}
z_0+\frac{a_{\omega}}{N}+\frac{1}{N^{\omega}(N+1)} \leq\sum_{j\in\Omega_1}\frac{ N^{\mathbf{s}_j} \mathbf{b}_{\mathfrak{n}_{j}}c_j l_j t_{k_2+i}}{b_{1}b_{2 } \cdots b_{k_2+i}} + \sum_{j\in\Omega_2}\frac{ N^{\mathbf{s}_j} \mathbf{b}_{\mathfrak{n}_{j}}c_j l_j t_{k_2+i}}{b_{1}b_{2 } \cdots b_{k_2+i}}\leq z_0+\frac{a_{\omega}+1}{N}-\frac{1}{N^{\omega}(N+1)}.
\end{align*}
Hence, \begin{equation}\label{X1}
\begin{split}
\frac{a_{\omega}}{N}+\frac{1}{N^{\omega}(N+1)} &\leq\sum_{j\in\Omega_1}\frac{ N^{\mathbf{s}_j} \mathbf{b}_{\mathfrak{n}_{j}}c_j l_j t_{k_2+i}}{b_{1}b_{2 } \cdots b_{k_2+i}} + \sum_{j\in\Omega_2}\frac{ N^{\mathbf{s}_j} \mathbf{b}_{\mathfrak{n}_{j}}c_j l_j t_{k_2+i}}{b_{1}b_{2 } \cdots b_{k_2+i}}-z_0
 \leq\frac{a_{\omega}+1}{N}-\frac{1}{N^{\omega}(N+1)}.
\end{split}
\end{equation}
Let $\beta$ be given in Lemma \ref{lem(4.4-1)} (i), we have $k_2+1< j_2+\beta$. Otherwise, by Lemma \ref{lem(4.4-1)} (i), we have $\max\{t'_n:n\geq 1\}\mathbf{b}_{\mathfrak{n}_{j_2}}<\mathbf{b}_{k_2+1}$. On the other hand,   $\max\{t_{k_2+i}': 1\leq i\leq\alpha\}\mathbf{b}_{\mathfrak{n}_{j_2}}\geq \mathbf{b}_{k_2+1}$ since $j_2\in\Omega_2$, which is a contradiction. Hence, $k_2+1< j_2+\beta$.
 It follows from that Lemma \ref{lem(4.4-1)} (ii),
 \begin{align*}
 \omega=
 \mathbf{s}_{k_2+i}+1-\mathbf{s}_{j_2}&=\tau( b_1b_2\cdots b_{k_2+i} )-\tau (t_{k_2+i})-\tau ( b_1b_2\cdots b_{j_2} )+\tau(t_{j_2})+1
 \\&\leq
 \tau( b_{j_2+1}b_{j_2+2}\cdots b_{k_2+i} )+\tau(t_{j_2})+1
\\&\leq (k_2-j_2+i+1)\max\{\tau(t_n): n\geq1\}+1
\\&\leq (\alpha+\beta)\max\{\tau(t_n): n\geq1\}+1:=\bf{\kappa_0}.
\end{align*}
Let
$$
W=\sum_{j\in\Omega_1}\frac{ N^{\mathbf{s}_j} \mathbf{b}_{\mathfrak{n}_{j}}c_j l_j t_{k_2+i}}{b_{1}b_{2 } \cdots b_{k_2+i}} + \sum_{j\in\Omega_2}\frac{ N^{\mathbf{s}_j} \mathbf{b}_{\mathfrak{n}_{j}}c_j l_j t_{k_2+i}}{b_{1}b_{2 }\cdots b_{k_2+i}}.
$$
According to \eqref{X1}, $$
W-z_0\in \left[\frac{a_{\omega}}{N}+\frac{1}{N^{\omega}(N+1)},\frac{a_{\omega}+1}{N}-\frac{1}{N^{\omega}(N+1)}\right].
$$
If $0\leq a_{\omega}\leq N-2$, for any $
l\in\mathcal{Z}(\hat{\delta}_{\{0,1,\cdots,N-1\}})=\frac{\mathbb{Z}\setminus N\mathbb{Z}}{N} $, we have
\begin{align*}
|W-l|=\left|W-z_0-(l-z_0)\right|&\geq \Big||l-z_0|-|W-z_0|\Big| \\
&\geq|l-z_0|-\left(\frac{a_{\omega}+1}{N}-\frac{1}{N^{\omega}(N+1)}\right)\\
&\geq\frac{1}{N^{\omega}(N+1)} \ \ (\text{ take} \ \ l=z_0+\frac{a_{\omega}+1}{N} )\\
&\geq \frac{1}{N^{\bf{\kappa_0}}(N+1)}.
\end{align*}
If $a_{\omega}=N-1$, for any $
l\in\mathcal{Z}(\hat{\delta}_{\{0,1,\cdots,N-1\}})=\frac{\mathbb{Z}\setminus N\mathbb{Z}}{N} $, we have
\begin{align*}
|W-l|=\left|W-z_0-(l-z_0)\right|&\geq \Big||l-z_0|-|W-z_0|\Big|\\
&\geq\frac{1}{N^{\omega}(N+1)}\ \ (\text{take} \ \ l=1-\frac{1}{N}+z_0 )\\
&\geq \frac{1}{N^{\bf{\kappa_0}}(N+1)}.
\end{align*}
These show that  $$\min\left\{\left|\sum_{j\in\Omega_1}\frac{ N^{\mathbf{s}_j} \mathbf{b}_{\mathfrak{n}_{j}}c_j l_j t_{k_2+i}}{b_{1}b_{2 } \cdots b_{k_2+i}} + \sum_{j\in\Omega_2}\frac{ N^{\mathbf{s}_j} \mathbf{b}_{\mathfrak{n}_{j}}c_j l_j t_{k_2+i}}{b_{1}b_{2 } \cdots b_{k_2+i}}-l\right|:l\in \mathcal{Z}(\hat{\delta}_{\{0,1,\cdots,N-1\}})\right\}
\geq\frac{1}{N^{ \bf{\kappa_0}  }(N+1)}.$$
Since $\bf{\kappa_0}$ is a fixed constant, there exists $\epsilon_{i_1}>0$ such that  \begin{align}\label{eq(4.3-3-1-1--5)}\left| \hat{\delta}_{\{0,1, \cdots, N-1\}}\left(\frac{\sum_{j\in\Omega_1\cup\Omega_2} N^{\mathbf{s}_j} \mathbf{b}_{\mathfrak{n}_{j}}c_j l_j t_{k_2+i}}{b_{1}b_{2 } \cdots b_{k_2+i}} \right) \right| >\epsilon_{i_1}.
\end{align}

Let $  \tilde{\epsilon }:=\min\{\epsilon_{i_0}, \epsilon_{i_1}\}$. Note that $\tilde{\epsilon }$  is not dependent on $i$, $k_1$ and $k_2$. According to \eqref{eq(4.3-3-1-1--4)}, \eqref{eq(4.3-3-1-1--5)} and the analysis of (I)-(III), we conclude that
  $$\left| \hat{\delta}_{\{0,1, \cdots, N-1\}}\left(\frac{\lambda t_{k_{2+i}}}{b_{1}b_{2 } \cdots b_{k_{2+i}} }\right) \right| >  \tilde{\epsilon}$$
for any $\lambda\in\mathcal{B}_{k_1,k_2}$, and the proof of Claim 3 is complete.
\end{proof}
It follows from Claim 3 that $\prod_{i=1}^{\alpha} | \hat{\delta}_{b_{k_2 +1}^{-1}b_{k_2 +2}^{-1} \cdots b_{k_2 +i}^{-1}D_{k_{2}+i}} (\frac{\lambda}{b_{1}b_{2 } \cdots b_{k_2} } ) | >\tilde{\epsilon }^{\alpha}:=\epsilon_0.$
Combining this with Proposition \ref{prop(4.1)}, we deduce that the proof is complete.
\end{proof}
Having established the above preparations, we can now prove  Case II.
\begin{thm}\label{th(1.4_2)}
 Under the assumption of Theorem \ref{th(1.4)}, suppose that $\mathbf{s}_i\neq\mathbf{s}_j$ for all $i\neq j$ and there exists $k_0\geq m_0$ such that $\min\{\mathbf{s}_j : j> k\}< \max\{\mathbf{s}_k : j\leq k\}$ for all $k\geq k_0$,
then $\mu_{\{b_k\},\{D_k\}}$ is a spectral measure.
\end{thm} \begin{proof}
 Let $ \Lambda_{0,k_{0}}=\bigoplus_{j=1}^{k_{0}} (N^{\mathbf{s}_j} \mathbf{b}_{\mathfrak{n}_{j}} \{0, 1, 2, \cdots, N-1\} )$ and let $\sigma_0=\min\{\theta_0,\vartheta_0\}$,  where $\theta_0$ and $\vartheta_0$ are given by  Proposition \ref{prop(4.2)} and \ref{prop(4.4.2)}, respectively.
Making  $k_{1}> k_{0} $ satisfy
$( b_1b_2\cdots b_{k_{1}})^{-1}\Lambda_{0,k_{0}}\subset[ -\sigma_0, \sigma_0]$. According to Proposition \ref{prop(4.4.2)}, we can choose a appropriate $\mathcal{B}_{k_0,k_1}:=\bigoplus_{j=k_0+1}^{k_{1}}\left(N^{\mathbf{s}_j} \mathbf{b}_{\mathfrak{n}_{j}}c_{ j } \{0, 1, 2, \cdots, N-1\}\right)$ with $c_j\in\mathbb{Z}\setminus N\mathbb{Z}$ to make
 \begin{align}\label{eq(4.3-3-1-1--6-1)}
 \prod_{i=1}^{\alpha} \left| \hat{\delta}_{b_{k_{1} +1}^{-1}b_{k_{1} +2}^{-1} \cdots b_{k_{1} +i}^{-1}D_{k_{1}+i}}\left(\frac{\lambda_0+\lambda_1}{b_{1}b_{2 } \cdots b_{k_{1}} }\right) \right|>\epsilon_{0}.
 \end{align}
 for any $\lambda_0\in \Lambda_{0,k_{0}}$ and $\lambda_1\in\mathcal{B}_{k_0,k_{1}}$. It follows from Proposition \ref{prop(4.2)} that for any $\lambda_1\in\mathcal{B}_{k_0,k_{1}}$,  there exists an integer $\mathbf{k}_{1, \lambda_1}$ such that
\begin{align}\label{eq(4.3-3-1-1--7)}
|\hat{\nu}_{>(k_{1}+\alpha)}( \frac{\lambda_0+\lambda_1}{b_1b_2\cdots b_{k_{1}+\alpha}}+\mathbf{k}_{1, \lambda_1})|>C
\end{align}
 for any $\lambda_0\in \Lambda_{0,k_{0}}$, where $\mathbf{k}_{1,0}=0$.
Let
$\Lambda_{k_0,k_{1}}=\bigcup_{\lambda_1\in \mathcal{B}_{k_0,k_{1}}}(\lambda_1+b_1b_2\cdots b_{k_{1}+\alpha} \mathbf{k}_{1, \lambda_1} )$ and $\Lambda_{1}:=\Lambda_{0,k_{0}}+\Lambda_{k_0,k_{1}} $. This means that $\Lambda_{0, k_{0}}$ and  $\Lambda_{k_0, k_{1}}$ satisfy  $(i)-(iii)$ of \eqref{eq(4.3-2)}.
From Proposition \ref{prop(4.2.1)}, \eqref{eq(4.3-3-1-1--6-1)} and \eqref{eq(4.3-3-1-1--7)}, we obtain  $\Lambda_{1}$ is a spectrum of $\mu_{k_{1}}$ and
  \begin{align*}
\left|\hat{\nu}_{>k_{1}}\left( \frac{\lambda}{b_1b_2\cdots b_{k_{1}}}\right)\right|&=\prod_{i=1}^{\alpha} \left| \hat{\delta}_{b_{k_{1} +1}^{-1}b_{k_{1} +2}^{-1} \cdots b_{k_{1} +i}^{-1}D_{k_{1}+i}}\left(\frac{\lambda}{b_{1}b_{2 } \cdots b_{k_{1}} }\right) \right|\left|\hat{\nu}_{>(k_{1}+\alpha)}\left( \frac{\lambda}{b_1b_2\cdots b_{k_{1}+\alpha}}\right)\right|
\\&>C\epsilon_0
\end{align*}
for any $ \lambda\in\Lambda_{1}$.

Let $k_{2}$ be a positive integer that satisfy $k_{2}>k_{1}$ and
$( b_1b_2\cdots b_{k_{2}})^{-1}\Lambda_{1}\subset[ -\sigma_0, \sigma_0]$.
Similarly, we can choose a appropriate $\mathcal{B}_{k_1,k_2}:=\bigoplus_{j=k_1+1}^{k_{2}}\left(N^{\mathbf{s}_j} \mathbf{b}_{\mathfrak{n}_{j}}c_{ j } \{0, 1, 2, \cdots, N-1\}\right)$ with $c_j\in\mathbb{Z}\setminus N\mathbb{Z}$ to make
 \begin{align}\label{eq(4.3-3-1-1--6)}
 \prod_{i=1}^{\alpha} \left| \hat{\delta}_{b_{k_{2} +1}^{-1}b_{k_{2} +2}^{-1} \cdots b_{k_{2} +i}^{-1}D_{k_{2}+i}}\left(\frac{\lambda_1+\lambda_2}{b_{1}b_{2 } \cdots b_{k_{2}} }\right) \right|>\epsilon_{0}.
 \end{align}
 for any  $\lambda_1\in \Lambda_{1}$ and $\lambda_2\in\mathcal{B}_{k_1,k_{2}}$. And for any $\lambda_2\in\mathcal{B}_{k_{1},k_{2}}$,  there exists  an integer $\mathbf{k}_{2,\lambda_2}$ such that
$$
|\hat{\nu}_{>k_{2}+\alpha}( \frac{\lambda_1+\lambda_2}{b_1b_2\cdots b_{k_{2}+\alpha}}+\mathbf{k}_{2, \lambda_2})|>C
$$
for any $\lambda_1\in \Lambda_{1}$, where $\mathbf{k}_{2, 0}=0$. Let $ \Lambda_{k_{1}, k_{2}}=\bigcup_{\lambda\in \mathcal{B}_{k_{1}, k_{2}}}(\lambda+b_1b_2\cdots b_{k_{2}+\alpha} \mathbf{k}_{{2}, \lambda})$  and  $ \Lambda_{2}= \Lambda_{1}+ \Lambda_{k_{1}, k_{2}}$. By Proposition \ref{prop(4.2.1)},  we have $\Lambda_{2}$ is a spectrum of $\mu_{k_{2}}$ and $ \Lambda_{1}\subset\Lambda_{2}$.
Moreover,
\begin{align*}
 \left|\hat{\nu}_{>k_{2}}( \frac{\lambda}{b_1b_2\cdots b_{k_{2}}})\right|&=\prod_{i=1}^{\alpha} \left| \hat{\delta}_{b_{k_{1} +1}^{-1}b_{k_{2} +2}^{-1} \cdots b_{k_{1} +i}^{-1}D_{k_{2}+i}}\left(\frac{\lambda}{b_{1}b_{2 } \cdots b_{k_{2}} }\right) \right| \left|\hat{\nu}_{>(k_{2}+\alpha)}\left( \frac{\lambda}{b_1b_2\cdots b_{k_{2}+\alpha}} \right)\right|
\\&>C\epsilon_0
\end{align*}
for any $ \lambda\in\Lambda_{2}$.

Repeat this operation,  we can find a strictly  increasing sequence $\{k_{i}\}_{i=1}^{\infty}$ such that for any $i\geq1$, the following three statements hold: (i) $\Lambda_{i}\subset \Lambda_{i+1}$ ; (ii) $\Lambda_{i}$ is a spectrum of $\mu_{k_{i}}$; (iii)
$
|\hat{\nu}_{>k_{i}}( \frac{\lambda}{b_1b_2\cdots b_{k_{i}}})|>C\epsilon_0
$
for any $ \lambda\in\Lambda_{i}$.
Combining this with Theorem \ref{prop(4.2.2)}, we have
$\Lambda=\bigcup_{i=1}^{\infty}\Lambda_{i}$ is a spectrum of $\mu_{\{b_k\},\{D_k\}}.$ Thus the proof follows.
\end{proof}
\begin{proof}[\bf Proof of  Theorem \ref{th(1.4)}]
 $``(i)\Longrightarrow (ii)\Longleftrightarrow (iii)"$ is obtained directly from  Theorems  \ref{th(1.2-1)} and \ref{th(1.2)}.

 $``(ii)\Longrightarrow (i)"$  can be  derived from  Theorems \ref{th(1.4_1)} and \ref{th(1.4_2)}.
\end{proof}

At the end of this paper, we give the following two examples to show that the Case I and Case II do exist respectively.
\begin{exam}{\rm
Let $D_{2k-1}=\{0,1\}$, $D_{2k}=\{0,1\}4$ and $b_{k}=18$ for all $k\geq 1$. It is easy to  verify that  $ \mathbf{s}_{2k-1}=2(k-1)$, $ \mathbf{s}_{2k }=2(k-1)-1$ for all $k\geq 1$ and $\min\{\mathbf{s}_j : j> k\}> \max\{\mathbf{s}_j : j\leq k\}$ for all $k\in 2\mathbb{N}^+$.
It follows from Theorem \ref{th(1.4_1)} that $\mu_{\{b_k\},\{D_k\}}$ is  a spectral measure.}
\end{exam}
\begin{exam}{\rm
Let $D_{2k-1}=\{0,1\}$, $D_{2k}=\{0,1\}16$ and $b_{k}=18$ for all $k\geq 1$. It is easy to  verify that  $ \mathbf{s}_{2k-1}=2(k-1)$, $ \mathbf{s}_{2k }=2(k-1)-3$ for all $k\geq 1$ and $\min\{\mathbf{s}_j : j> k\}< \max\{\mathbf{s}_j : j\leq k\}$ for all $k\geq 1$.
It follows from  Theorem \ref{th(1.4_2)} that $\mu_{\{b_k\},\{D_k\}}$ is a spectral measure.}
\end{exam}
\noindent \textbf{\bf Conflict of interest.} We declare that we do not have any commercial or associative interest that represents a conflict of interest in connection with the work submitted.

\end{document}